\documentclass{amsart}
\usepackage{amsthm,amsfonts,amssymb,amsmath,amscd
 ,newclude
}
\usepackage{textcomp}
\usepackage{appendix}
\usepackage{xcolor}
\usepackage{longtable}
\usepackage{url}
\usepackage[centerlast]{subfigure}
\usepackage[all,2cell, cmtip]{xy} \UseAllTwocells

\begin{document}

\title[A higher Grothendieck construction]{A higher Grothendieck construction}
\author[A. Sharma]{Amit Sharma}
\email{asharm24@kent.edu}
\address {Department of mathematical sciences\\ Kent State University\\
  Kent,OH}

\date{July 18, 2020}
%

\newcommand{\CONT}{\noindent}
\newcommand{\FIG}{Fig.\ }
\newcommand{\FIGS}{Figs.\ }
\newcommand{\SEC}{Sec.\ }
\newcommand{\SECS}{Secs.\ }
\newcommand{\TAB}{Table }
\newcommand{\TABS}{Tables }
\newcommand{\EQ}{Eq.\ }
\newcommand{\EQS}{Eqs.\ }
\newcommand{\APP}{Appendix }
\newcommand{\APPS}{Appendices }
\newcommand{\CHP}{Chapter }
\newcommand{\CHPS}{Chapters }

\newcommand{\OFF}{\emph{G2off}~}
\newcommand{\TOO}{\emph{G2Too}~}
\newcommand{\CatS}{Cat_{\bigS}}
\newcommand{\PicS}{{\underline{\pic}}^{\oplus}}
\newcommand{\HPicS}{{Hom^{\oplus}_{\pic}}}

\newtheorem{thm}{Theorem}[section]
\newtheorem{lem}[thm]{Lemma}
\newtheorem{conj}[thm]{Conjecture}
\newtheorem{coro}[thm]{Corollary}
\newtheorem{prop}[thm]{Proposition}

\theoremstyle{definition}
\newtheorem{df}[thm]{Definition}
\newtheorem{nota}[thm]{Notation}

\newtheorem{ex}[thm]{Example}
\newtheorem{exs}[thm]{Examples}

\theoremstyle{remark}
\newtheorem*{note}{Note}
\newtheorem{rem}{Remark}
\newtheorem{ack}{Acknowledgments}

\newcommand{\ChI}{{\textit{\v C}}\textit{ech}}
\newcommand{\Ch}{{\v C}ech}

\newcommand{\ChZG}{hermitian line $0$-gerbe}
\renewcommand{\theack}{$\! \! \!$}

\newcommand{\ChG}{flat hermitian line $1$-gerbe}
\newcommand{\ChC}{hermitian line $1$-cocycle}
\newcommand{\ChGG}{flat hermitian line $2$-gerbe}
\newcommand{\ChCC}{hermitian line $2$-cocycle}
\newcommand{\id}{id}
\newcommand{\LC}{\mathfrak{C}}
\newcommand{\Coker}{Coker}
\newcommand{\Com}{Com}
\newcommand{\Hom}{Hom}
\newcommand{\Mor}{Mor}
\newcommand{\Map}{Map}
\newcommand{\alg}{alg}
\newcommand{\an}{an}
\newcommand{\Ker}{Ker}
\newcommand{\Ob}{Ob}
\newcommand{\Proj}{\mathbf{Proj}}
\newcommand{\topo}{\mathbf{Top}}
\newcommand{\kan}{\mathcal{K}}
\newcommand{\pkan}{\mathcal{K}_\bullet}
\newcommand{\Kan}{\mathbf{Kan}}
\newcommand{\pKan}{\mathbf{Kan}_\bullet}
\newcommand{\QCat}{\mathbf{QCat}}
\newcommand{\gp}{\mathcal{A}_\infty}
\newcommand{\mdl}{\mathcal{M}\textit{odel}}
\newcommand{\sSets}{\mathbf{\S}}
\newcommand{\Sets}{\mathbf{Sets}}
\newcommand{\sSetsM}{\mathbf{\S}^+}
\newcommand{\sSetsQ}{(\mathbf{\sSets, Q})}
\newcommand{\sSetsMQ}{(\mathbf{\sSetsM, Q})}
\newcommand{\sSetsK}{(\mathbf{\sSets, \Kan})}
\newcommand{\pSSets}{\mathbf{\sSets}_\bullet}
\newcommand{\pSSetsK}{(\mathbf{\sSets}_\bullet, \Kan)}
\newcommand{\pSSetsQ}{(\mathbf{\sSets_\bullet, Q})}
\newcommand{\cyl}{\mathbf{Cyl}}
\newcommand{\lin}{\mathcal{L}_\infty}
\newcommand{\Vect}{\mathbf{Vect}}
\newcommand{\Aut}{Aut}
\newcommand{\pic}{\mathcal{P}\textit{ic}}
\newcommand{\Dlin}{\pic}
\newcommand{\bigS}{\mathbf{S}}
\newcommand{\bigA}{\mathbf{A}}
\newcommand{\bhom}{\mathbf{hom}}
\newcommand{\bhomK}{\mathbf{hom}({\textit{K}}^+,\textit{-})}
\newcommand{\Bhom}{\mathbf{Hom}}
\newcommand{\bhomk}{\mathbf{hom}^{{\textit{k}}^+}}
\newcommand{\Dlino}{\pic^{\textit{op}}}
\newcommand{\lino}{\mathcal{L}^{\textit{op}}_\infty}
\newcommand{\lind}{\mathcal{L}^\delta_\infty}
\newcommand{\linK}{\mathcal{L}_\infty(\kan)}
\newcommand{\linC}{\mathcal{L}_\infty\text{-category}}
\newcommand{\linCs}{\mathcal{L}_\infty\text{-categories}}
\newcommand{\ainCs}{\text{additive} \ \infty-\text{categories}}
\newcommand{\ainC}{\text{additive} \ \infty-\text{category}}
\newcommand{\inC}{\infty\text{-category}}
\newcommand{\inCs}{\infty\text{-categories}}
\newcommand{\gS}{{\Gamma}\text{-space}}
\newcommand{\gSet}{{\Gamma}\text{-set}}
\newcommand{\ggS}{\Gamma \times \Gamma\text{-space}}
\newcommand{\gSs}{\Gamma\text{-spaces}}
\newcommand{\gSets}{\Gamma\text{-sets}}
\newcommand{\ggSs}{\Gamma \times \Gamma\text{-spaces}}
\newcommand{\gO}{\Gamma-\text{object}}
\newcommand{\gSCat}{{\Gamma}\text{-space category}}
\newcommand{\pss}{\mathbf{S}_\bullet}
\newcommand{\gSC}{{{{\Gamma}}\mathcal{S}}}
\newcommand{\gSCM}{{{{\Gamma}}\mathcal{S}^+}}
\newcommand{\pGSC}{{{{\Gamma}}\mathcal{S}}_\bullet}
\newcommand{\pGSCStr}{{{{\Gamma}}\mathcal{S}}_\bullet^{\textit{str}}}
\newcommand{\ggSC}{{\Gamma\Gamma\mathcal{S}}}
\newcommand{\gSD}{\mathbf{D}(\gSC^{\textit{f}})}
\newcommand{\sCat}{\mathbf{sCat}}
\newcommand{\pSCat}{\mathbf{sCat}_\bullet}
\newcommand{\gSetCat}{{{{\Gamma}}\mathcal{S}\textit{et}}}
\newcommand{\Dhom}{\mathbf{R}Hom_{\pic}}
\newcommand{\gop}{\Gamma^{\textit{op}}}
\newcommand{\fU}{\mathbf{U}}
\newcommand{\cDN}{\underset{\mathbf{D}[\textit{n}^+]}{\circ}}
\newcommand{\cDK}{\underset{\mathbf{D}[\textit{k}^+]}{\circ}}
\newcommand{\cDL}{\underset{\mathbf{D}[\textit{l}^+]}{\circ}}
\newcommand{\cD}{\underset{\gSD}{\circ}}
\newcommand{\cDT}{\underset{\gSD}{\widetilde{\circ}}}
\newcommand{\ppsSets}{\sSets_{\bullet, \bullet}}
\newcommand{\gdHom}{\underline{Hom}_{\gSD}}
\newcommand{\HomU}{\underline{Hom}}
\newcommand{\ominf}{\Omega_\infty}
\newcommand{\ev}{ev}
\newcommand{\cu}{C(X;\mathfrak{U}_I)}
\newcommand{\Sing}{Sing}
\newcommand{\AlgEin}{\A\textit{lg}_{\E_\infty}}
\newcommand{\SFunc}[2]{\mathbf{SFunc}({#1} ; {#2})}
\newcommand{\unit}[1]{\mathrm{1}_{#1}}
\newcommand{\liminj}{\varinjlim}
\newcommand{\limproj}{\varprojlim}
\newcommand{\HMapC}[3]{\mathcal{M}\textit{ap}^{\textit{h}}_{#3}(#1, #2)}
\newcommand{\tensPGSR}[2]{#1 \underset{\gSR}\wedge #2}
\newcommand{\pTensP}[3]{#1 \underset{#3}\wedge #2}
\newcommand{\MGCat}[2]{\underline{\map}_{\gSC}({#1},{ #2})}
\newcommand{\MGBoxCat}[2]{\underline{\map}_{\gSC}^{\Box}({#1},{ #2})}
\newcommand{\TensPFunc}[1]{- \underset{#1} \otimes -}
\newcommand{\TensP}[3]{#1 \underset{#3}\otimes #2}
\newcommand{\MapC}[3]{\mathcal{M}\textit{ap}_{#3}(#1, #2)}
\newcommand{\bHom}[3]{{#2}^{#1}}
\newcommand{\gn}[1]{\Gamma^{#1}}
\newcommand{\gnk}[2]{\Gamma^{#1}({#2}^+)}
\newcommand{\gnf}[2]{\Gamma^{#1}({#2})}
\newcommand{\ggn}[1]{\Gamma\Gamma^{#1}}
\newcommand{\Nat}{\mathbb{N}}
\newcommand{\partition}[2]{\delta^{#1}_{#2}}
\newcommand{\inclusion}[2]{\iota^{#1}_{#2}}
\newcommand{\EinQC}{\text{coherently commutative monoidal quasi-category}} 
\newcommand{\EinQCs}{\text{coherently commutative monoidal quasi-categories}}
\newcommand{\pHomCat}[2]{[#1,#2]_{\bullet}}
\newcommand{\CatHom}[3]{[#1,#2]^{#3}}
\newcommand{\pCatHom}[3]{[#1,#2]_\bullet^{#3}}
\newcommand{\EinC}{\text{coherently commutative monoidal category}}
\newcommand{\EinCs}{\text{coherently commutative monoidal categories}}
\newcommand{\EinLO}{E_\infty{\text{- local object}}}
\newcommand{\EinSLO}{\E_\infty\S{\text{- local object}}}
\newcommand{\Ein}{E_\infty}
\newcommand{\EinS}{E_\infty{\text{- space}}}
\newcommand{\EinSs}{E_\infty{\text{- spaces}}}
\newcommand{\PCat}{\mathbf{Perm}}
\newcommand{\nor}[1]{{#1}^\textit{nor}}
\newcommand{\pSSetsHom}[3]{[#1,#2]_\bullet^{#3}}
\newcommand{\PNat}{\overline{\L}}
\newcommand{\PStr}{\L}
\newcommand{\Gn}[1]{\Gamma[#1]}
\newcommand{\GIH}{\Gamma\textit{H}_{\textit{in}}}
\newcommand{\QStr}[1]{\L_\bullet(\ud{#1})}
\newcommand{\QStrF}{\L_\bullet}
\newcommand{\Kbar}{\overline{\K}}
\newcommand{\gPerm}{{\Gamma\PCat}}
\newcommand{\gCat}{{\Gamma\Cat}}
\newcommand{\MapS}[3]{\map_{#3}(#1, #2)}
\newcommand{\sSetsMG}{\sSetsM / N(\gop)}
\newcommand{\sSetsMGen}[1]{\sSetsM / N(#1)}
\newcommand{\pF}{\mathfrak{F}_\bullet^+(\gop)}
\newcommand{\pN}{{\textit{N}}_\bullet^+(\gop)}
\newcommand{\pFX}[1]{\mathfrak{F}_{#1}^+(\gop)}
\newcommand{\pNX}[1]{{\textit{N}}_{#1}^+(\gop)}
\newcommand{\nGop}{N(\gop)}
\newcommand{\sSetsMGSM}{(\sSetsM/ N(\gop), \otimes)}
\newcommand{\sSetsGen}[1]{\sSets/ #1}
\newcommand{\ovCatGen}[2]{#1/ #2}
\newcommand{\coMdl}[1]{(\sSets/ #1, \mathbf{L})}
\newcommand{\gCLM}[1]{\mathfrak{L}^+_{#1}}
\newcommand{\gCL}[1]{\mathfrak{L}_{#1}}
\newcommand{\pFGen}[2]{\mathfrak{F}_{#1}^+(#2)}

\def\Pic{\mathbf{2}\mathcal P\textit{ic}}
\def\nc{\mathbb C}

\def\Z{\mathbb Z}
\def\P{\mathbb P}
\def\J{\mathcal J}
\def\I{\mathcal I}
\def\nC{\mathbb C}
\def\H{\mathcal H}
\def\A{\mathcal A}
\def\C{\mathcal C}
\def\D{\mathcal D}
\def\E{\mathcal E}
\def\G{\mathcal G}
\def\B{\mathcal B}
\def\L{\mathcal L}
\def\U{\mathcal U}
\def\K{\mathcal K}
\def\El{\mathcal E{\textit{l}}}

\def\M{\mathcal M}
\def\O{\mathcal O}
\def\R{\mathcal R}
\def\S{\mathcal S}
\def\N{\mathcal N}

\newcommand{\undertilde}[1]{\underset{\sim}{#1}}
\newcommand{\abs}[1]{{\lvert#1\rvert}}
\newcommand{\mC}[1]{\mathfrak{C}(#1)}
\newcommand{\sigInf}[1]{\Sigma^{\infty}{#1}}
\newcommand{\x}[4]{\underset{#1, #2}{ \overset{#3, #4} \prod }}
\newcommand{\mA}[2]{\textit{Add}^n_{#1, #2}}
\newcommand{\mAK}[2]{\textit{Add}^k_{#1, #2}}
\newcommand{\mAL}[2]{\textit{Add}^l_{#1, #2}}
\newcommand{\Mdl}[2]{\L_\infty}
\newcommand{\inv}[1]{#1^{-1}}
\newcommand{\Lan}[2]{\mathbf{Lan}_{#1}(#2)}

\newcommand{\del}{\partial}
\newcommand{\sCatO}{\mathcal{S}Cat_\O}
\newcommand{\FCgop}{\mathbf{F}\mC{N(\gop)}}
\newcommand{\hProd}{{\overset{h} \oplus}}
\newcommand{\hProdn}{\underset{n}{\overset{h} \oplus}}
\newcommand{\hProdk}[1]{\underset{#1}{\overset{h} \oplus}}
\newcommand{\map}{\mathcal{M}\textit{ap}}
\newcommand{\SMGS}[2]{\map_{\gSC}({#1},{ #2})}
\newcommand{\MGS}[2]{\underline{\map}_{\gSC}({#1},{ #2})}
\newcommand{\MGSBox}[2]{\underline{\map}^{\Box}_{\gSC}({#1},{ #2})}
\newcommand{\Aqcat}[1]{\underline{#1}^\oplus}
\newcommand{\Cat}{\mathbf{Cat}}
\newcommand{\Sp}{\mathbf{Sp}}
\newcommand{\SpStb}{\mathbf{Sp}^{\textit{stable}}}
\newcommand{\SpStr}{\mathbf{Sp}^{\textit{strict}}}
\newcommand{\Sspec}{\mathbb{S}}
\newcommand{\ud}[1]{\underline{#1}}
\newcommand{\inrt}{\mathbf{Inrt}}
\newcommand{\act}{\mathbf{Act}}
\newcommand{\StrSMHom}[2]{[#1,#2]_\otimes^{\textit{str}}}
\newcommand{\Sh}[1]{{#1}^\sharp}
\newcommand{\Fl}[1]{{#1}^\flat}
\newcommand{\Nt}[1]{{#1}^\natural}
\newcommand{\Flmap}[3]{\Fl{\left[#1, #2\right]}_{#3}}
\newcommand{\Shmap}[3]{\Sh{\left[#1, #2 \right]}_{#3}}
\newcommand{\mRN}[1]{\int^{n^+ \in \gop}{#1}^+}
\newcommand{\mRNGen}[2]{\int_{+}^{d \in {#2}}{#1}}
\newcommand{\mRNL}[1]{{\mathfrak{F}}^+_{#1}(\gop)}
\newcommand{\mRNGenL}[2]{{\mathfrak{F}}^+_{#1}(#2)}
\newcommand{\mapG}[2]{[#1, #2]_{\gop}^+}
\newcommand{\mapFl}[2]{\Fl{[#1, #2]}}
\newcommand{\mapSh}[2]{\Sh{[#1, #2]}}
\newcommand{\mapMS}[2]{[#1, #2]^+}
\newcommand{\ExpG}[2]{{\left({#2}\right)}^{[#1]}}
\newcommand{\expG}[2]{{{#2}}^{[#1]}}
\newcommand{\mapGen}[3]{[#1, #2]_{#3}}
\newcommand{\mapGenM}[3]{[#1, #2]^+_{#3}}
\newcommand{\rNGen}[2]{\int^{d \in {#2}}{#1}}
\newcommand{\rNGenL}[2]{\mathfrak{L}_{#2}(#1)}
\newcommand{\rN}[1]{\int^{n^+ \in \gop} {#1}}
\newcommand{\rNGenR}[1]{\mathfrak{R}_{#1}}
\newcommand{\mRNGenR}[1]{\mathfrak{R}^{+}_{#1}}
\newcommand{\cn}[1]{\mathbf{id}(d)_{{#1}}}
\newcommand{\cnGen}[2]{\mathbf{id}(#2)_{{#1}}}
\newcommand{\piSSGen}[1]{\left( \sSetsMGen{#1} \right)^{\pi_0}}
\newcommand{\colim}[1]{{\varinjlim}^{#1}}
\newcommand{\hColim}[1]{{\varinjlim}^{\textit{h}}_{#1}}

\begin{abstract}
 The main objective of this paper is to construct a homotopy colimit functor on a category of functors taking values in the model category of quasi-categories.
\end{abstract}

\maketitle

\tableofcontents

\section[Introduction]{Introduction}
\label{Introduction}
The Grothendieck construction is ubiquitous in category theory. This construction associates to a (pseudo) functor $F:D \to \Cat$, a (op)fibration over the (small) category $D$. The construction  establishes an equivalence and therefore allows us to switch between $\Cat$-valued functors and fibrations. In this paper we want to extend the classical Grothendieck construction to functors taking value in the model category of quasi-categories with the aim of establishing an equivalence between a category of $\sSets$-valued functors and appropriately defined (simplicial) fibrations over the nerve of the domain category of functors. The main objective of this paper is to use the aforementioned equivalence to construct a \emph{homotopy colimit} functor for functors taking value in the the model category of quasi-categories. We will be primarily working with the adaptation of the model category structure of quasi-categories on $\sSets$, \cite{AJ1}, \cite{AJ2}, to marked simplicial sets which is reviewed in appendix \ref{mar-sSets}.

It is well known that a \emph{left fibration} of simplicial sets over the nerve of a (small) category $N(D)$ is determined upto equivalence by a homotopy coherent diagram taking values in (a higher category of) Kan complexes, see \cite[5.3]{DC}. The same holds for \emph{coCartesian fibration} of simplicial sets over $N(D)$ with respect to homotopy coherent diagrams taking values in (a higher category of) quasi-categories, see \cite[Ch. 3]{JL}. In this paper we show that the aforementioned homotopy coherent diagrams can be rectified \emph{i.e.} upto equivalence they can be replaced by an honest functor. More precisely, we will show that for each coCartesian fibration $p:X \to N(D)$, there exists a (honest) functor $Z:D \to \sSets$, taking values in quasi-categories whose \emph{Grothendieck construction}, denoted $\rNGen{Z}{D}$, is equivalent to the fibration $p$ in a suitably defined model category structure on $\sSetsMGen{D}$. Such a result first appeared 
in \cite[Ch. 3]{JL} where the author defines an extension of the classical Grothendieck construction called \emph{(marked) relative nerve} which determines a functor $N^+_\bullet(D):[D, \sSetsM] \to \sSetsMGen{N(D)}$. This functor is shown to be the right Quillen functor of a Quillen adjunction between the \emph{coCartesian} model category structure on $\sSetsMGen{D}$ and the projective model category structure on $[D, \sSetsMQ]$. In this paper we construct a simplicial space (bisimplicial set) from a simplicial sets valued functor. We show that the zeroth row of this bisimplicial set is isomorphic to the relative nerve of the original functor. Our approach to this construction is to directly extend the Grothendieck construction of a $\Cat$ valued functor to functors taking value in (a category of) quasi-categories. Our main objective in this paper is to use the (higher) Grothendieck construction to construct a homotopy colimit functor. The drived functor of a homotopy colimit functor is a left adjoint which prompts us to write an \emph{approximation} to the (marked) relative nerve functor which is a left Quillen functor. Such an approximation to the (unmarked) relative nerve functor was constructed in \cite{HM} which is shown to be a Quillen equivalence between the \emph{covariant} model category structure on $\sSetsGen{N(D)}$, see \cite[Ch. 8]{AJ1} and the \emph{projective} model category structure on the functor category $[D, \sSetsK]$.

A colimit functor, defined on a category of functors $[D, \M]$ which inherits a \emph{projective} model category structure from the model category $\M$, is a left Quillen functor. However it is usually not a \emph{homotopical} functor.
A homotopy colimit functor is a homotopical functor whose derived functor is isomorphic to a left derived functor of a colimit functor. The homotopy colimit of a $\Cat$ valued functor is closely related to the Grothendieck construction. A homotopy colimit of a functor taking values in the Thomason model category of (small) categories, see \cite{Thomason2}, was defined in \cite{Thomason3}. In the same paper it was shown that the nerve of the Grothendieck construction is weakly equivalent to its homotopy colimit. The homotopy colimit of a functor taking values in the \emph{natural} model category of (small) categories is also related to the Grothendieck construction. In this case the homotopy colimit is obtained by inverting the \emph{horizontal} arrows of the Grothendieck construction of the functor. It is the author's understanding that this construction of the homotopy colimit first appeared in the works of \emph{Giraud}. The main objective of this paper is to extend Giraud's construction to functors taking value in the model category of quasi-categories. More precisely, we show that the homotopy colimit of such a functor is obtained by inverting the \emph{coCartesian} edges of its (higher) Grothendieck construction.

   \begin{ack}
   	The author would like to thank Andre Joyal for having useful discussions on the subject and for sharing his views on the notion of a higher Gothendieck construction.
   	\end{ack}
 \section{A higher Grothendieck construction }
\label{inf-cat-Gr-const}
In this section we will describe a \emph{higher Grothendieck construction} for quasi-categories. The classical Grothendieck construction defines a functor
\[
\rNGen{-}{D}:[D; \Cat] \to \ovCatGen{\Cat}{D} \overset{N} \to \ovCatGen{\sSets}{N(D)}
\]
We want to construct a \emph{left Kan extension} of the above functor along the Nerve functor $[D; N]:[D; \Cat] \to [D; \sSets]$ which we call a \emph{higher Grothendieck construction}. Our approach will be to directly extend the classical Grothendieck construction to the category of functors taking value in quasi-categories.

 Let $X:D \to \sSets$ be a functor. We recursively define a collection of simplicial sets as follows:
 \[
 \G^X_0(d) := X(d).
 \]
 For a map $f:d_1 \to d_2$ in $D$, we define a simplicial set $\G_1(f)$ by the following pullback square:
 \[
 \xymatrix@C=18mm{
 \G^X_1(f) \ar[r]^{p_2(f)} \ar[d]_{p_1(f)} & [\Delta[1]; X(d_2)] \ar[d]^{[d_1; X(d_2)] \times [d_0; X(d_2)]} \\
 [\Delta[0]; X(d_1)] \times [\Delta[0]; X(d_2)] \ar[r]_{ [\Delta[0]; X(f)] \times id} & [\Delta[0]; X(d_2)] \times [\Delta[0]; X(d_2)]
 }
 \]
 \begin{rem}
 	For each object $d \in D$ 
 	\[
 	\G^X_1(id_d) = [\Delta[1], X(d)].
 	\]
 	\end{rem}
 For a pair of maps $f_1:d_1 \to d_2$, $f_2:d_2 \to d_3$ in $D$, we define a simplicial set $\G^X_2(f_1, f_2)$ by the following pullback square:
  \[
 \xymatrix@C=18mm{
 	\G^X_2(f_1, f_2) \ar[r]^{p_2((f_1, f_2))} \ar[d]_{p_1((f_1, f_2))} & [\Delta[2]; X(d_3)] \ar[d]^{( [d_0; X(d_3)], [d_1; X(d_3)], [d_2; X(d_3)] ) } \\
 	\G_1^X(f_2) \times \G^X_1(f_2 f_1) \times  \G_1^X(f_1)  \ar[r]_{\ \ \ \ \ \ \ \ F_3 \times F_2 \times F_1 } &  \underset{3} \prod [\Delta[1]; X(d_3)] 
 }
 \]
 where $F_1$ is the composite map:
 \[
 \G_1^X(f_1) \overset{p_2(f_1)} \to [\Delta[1]; X(d_2)] \overset{[\Delta[1]; X(f_2)]} \to [\Delta[1]; X(d_3)],
 \]
 $F_3 = p_2(f_2)$, $F_2 = p_2(f_2f_1)$  and $p_1(f_1, f_2) = (p_1(f_1), p_1(f_2), p_1(f_2f_1))$.
 \begin{rem}
 	For each $f \in Mor(D)$, the simplicial sets $\G^X_2((f, id))$ and $\G^X_2((id, f))$ are given by the following two pullback squares respectively:
 	\[
 	\xymatrix{
 		\G^X_2(f, id) \ar[r]^{p_2((f, id))} \ar[d]_{p_1((f, id))} & [\Delta[2]; X(d_2)] \ar[d]^{( [d_1; X(d_3)], [d_2; X(d_2)] ) }  \\
 		 \G^X_1(f) \times  \G_1^X(f)  \ar[r]_{ F_2 \times F_1 \ \ \ \  \ \ \ \ \ } & [\Delta[1]; X(d_2)] \times [\Delta[1]; X(d_2)] 
    }
 	\]
 	and
 	\[
 	\xymatrix@C=16mm{
 		\G^X_2(id, f) \ar[r]^{p_2((id, f))} \ar[d]_{p_1((id, f))} & [\Delta[2]; X(d_2)] \ar[d]^{( [d_0; X(d_2)], [d_1; X(d_2)], [d_2; X(d_2)] ) }  \\
 		\G^X_1(f) \times  \G_1^X(f) \times [\Delta[1], X(d_1)]  \ar[r]_{ F_3 \times F_2 \times F_1 \ \ \ \ \ \ \  } & [\Delta[1]; X(d_2)] \times [\Delta[1]; X(d_2)] \times [\Delta[1]; X(d_2)] 
 	}
 	\]
 	\end{rem}
For an $n$-tuple $\sigma =(f_1, f_2, \dots, f_n) \in (N(D))_n$, we define a simplicial set $\G^X_n(\sigma)$ by the following pullback square:

\begin{equation}
\label{P-n}
\xymatrix@C=22mm{
	\G^X_n(\sigma) \ar[r]^{p_2(\sigma)} \ar[d]_{p_1(\sigma)} & [\Delta[n]; X(d_{n+1})] \ar[d]^{H } \\
  \overset{n}{\underset{i=0} \prod}\G_{n-1}^X(d_i(\sigma))  \ar[r]_{ F_{n+1} \times \cdots F_1 \ \ } &  \underset{n} \prod [\Delta[n-1]; X(d_{n+1})] 
}
\end{equation}
where $H = ([d_{0}; X(d_{n+1})], [d_{1}; X(d_{n+1})], \dots,  [d_{n}; X(d_{n+1})])$ and for $2 \le i \le n+1$ the simplicial map $F_i$ is the following composite:
 \[
\G_{n-1}^X(d_{i}(\sigma)) \overset{p_2(d_{i}(\sigma))} \to [\Delta[n-1];X(d_{n+1})] 
 \]
The map $F_1$ is the following composite
\[
\G_{n-1}^X(d_{n}(\sigma)) \overset{p_2(d_n(\sigma))} \to [\Delta[n-1];X(d_{n})] \overset{[\Delta[n-1];X(f_n )]} \to [\Delta[n-1];X(d_{n+1})]
 \]
 \begin{rem}
 	\label{can-simplex-Gerbe}
 	For the canonical simplex $\sigma = \cn{n}$, see definition \ref{can-simplices}, the simplicial set
 	\[
 	\G^X_n(\cn{n}) = [\Delta[n], X(d_{n+1})].
 	\]
 	\end{rem}
\begin{df}
	\label{One-Gerbe-over-n-simplex}
	For a pair consisting of an $n$-simplex $\sigma \in N(D)_n$ and a functor $X:D \to \sSets$, we will refer to $\G^X_n(\sigma)$ as the $1$-\emph{Gerbe over }$\sigma$ determined by $X$.
	\end{df}

\begin{prop}
	\label{inc-degen-high-fibs}
	For each $(n-1)$-simplex $\rho$ in $N(D)$ there is an inclusion map
	\[
	\iota^j_\rho:\G_{n-1}^X(\rho) \to \G_{n}^X(s_j(\rho))
	\]
	where $s_j$ is the $j$th degeneracy operator of $N(D)$ for $1 \le j \le n$. 
	\end{prop}
\begin{proof}
	The simplicial map $\iota^j_\rho$ is the unique map into the pullback shown in the following diagram:
	\begin{equation*}
	\xymatrix@C=16mm{
	\G_{n-1}^X(\rho) \ar[rr]^{p_2(\rho)} \ar@/_2pc/[rdd]_{i_j} \ar[rd]^{\iota^j_\rho} &&  [\Delta[n-1]; X(d_{n+1})] \ar[d]^{[s_j, X(d_{n+1})]} \\
	&\G^X_n(s_j(\rho)) \ar[r]^{p_2(s_j(\rho))} \ar[d]_{p_1(s_j(\rho))} & [\Delta[n]; X(d_{n+1})] \ar[d]^{H } \\
	&\overset{n}{\underset{i=0} \prod}\G_{n-1}^X(d_i(s_j(\rho)))  \ar[r]_{ F_{n+1} \times \cdots F_1 \ \ } &  \underset{n} \prod [\Delta[n-1]; X(d_{n+1})] 
     }
	\end{equation*}
	where $i_j$ is the inclusion into the $j$th component namely $\G_{n-1}^X(d_js_j(\rho)) = \G_{n-1}^X(\rho)$.
	\end{proof}
\begin{prop}
There is a simplicial space i.e. a functor $\left(\rNGen{X}{D}\right)_\bullet:\Delta^{op} \to \sSets$ whose degree $n$ simplicial-set is defined as follows:
\[
\left(\rNGen{X}{D}\right)_\bullet([n]):= \underset{\sigma \in (N(D))_n} \sqcup \lbrace \sigma \rbrace \times \G^X_n(\sigma)
\]
 \end{prop}
\begin{proof}
	We will define the degeneracy and face operators. Each $\G_n^X(\sigma)$ is equipped with a projection map
	\[
	d_i(p_1(\sigma)):\G_n^X(\sigma) \to \G_{n-1}^X(d_i(\sigma)) 
	\]
	For $i \in \lbrace 0, 1, 2, \dots,  n \rbrace$, this map is given by the following composite:
	\[
	\G^X_n(\sigma) \overset{p_1(\sigma)} \to \overset{n}{\underset{i=0} \prod}\G_{n-1}^X(d_i(\sigma)) \overset{pr_i} \to \G_{n-1}^X(d_i(\sigma)),
	\]
	where $f_n:d_n \to d_{n+1}$ is the last map in $\sigma = (f_1, \dots, f_n)$ and $pr_i$ are the obvious projections from the product. The maps $d_i(p_1(\sigma))$ join together to form a map
	\[
	d_i:\underset{\sigma \in (N(D))_n} \sqcup  \G^X_n(\sigma) \to \underset{\rho \in (N(D))_{n-1}} \sqcup  \G^X_{n-1}(\sigma)
	\]
	which is our $ith$ face operator for $0 \le i \le n$.
	
	The maps $\iota^j_\rho$ from proposition \ref{inc-degen-high-fibs} gives us the $ith$ degeneracy map
	\[
	s_j:\underset{\rho \in (N(D))_{n-1}} \sqcup  \G^X_{n-1}(\rho) \to \underset{\sigma \in (N(D))_{n}} \sqcup  \G^X_{n}(\rho)
	\]
	
	\end{proof}
\begin{nota}
	\label{Box-prod-sSets}
	Each pair $(K, L)$ of simplicial sets defines a \emph{bisimplicial sets} \emph{i.e.} a functor
	\[
	K \Box L: \Delta^{op} \times \Delta^{op}  \to \Sets
	 \]
	as follows:
	\[
	K \Box L([m],[n]) := K_m \times L_n
	\]
	\end{nota}
\begin{rem}
	\label{proj-over-Ner}
	The simplicial space $\left( \rNGen{X}{D} \right)_\bullet$ is equipped with a map of simplicial spaces:
	\[
	p^X_\bullet:\left( \rNGen{X}{D} \right)_\bullet \to N(D) \Box \Delta[0].
	\]
	\end{rem}
\begin{nota}
Each simplicial space $Z:\Delta^{op} \to \sSets$ determines a bisimplicial set, also denoted by $Z$

\[
Z:\Delta^{op} \times \Delta^{op} \to \Sets
\]
by $Z([m], [n]) = (Z[m])_n$. Further we denote the following simplicial set by $i_1^\ast(Z)$:
\[
\Delta \overset{(-, [0])}\to \Delta \times \Delta \overset{Z} \to \Sets
\]
\end{nota}
Now we can define the (total space of) the Grothendieck construction of $X:D \to \sSets$ as follows:
\begin{equation}
\rNGen{X}{D} = i^\ast_1 \left(\left( \rNGen{X}{D} \right)_\bullet\right)
\end{equation}
\begin{rem}
	\label{rep-Gr-cons}
	The set of $n$-simplices of $\rNGen{X}{D}$ can be represented as follows:
	\[
	\left( \rNGen{X}{D} \right)_n = \underset{\sigma \in N(D)_n} \sqcup \lbrace \sigma \rbrace \times \left( \G^X_n(\sigma)\right)_0.
	\]
	\end{rem}
\begin{rem}
	\label{desc-n-simp-Gr}
	An $n$-simplex $\delta$ of $\rNGen{X}{D}$ is a pair $\delta = (\sigma, \beta)$ where $\sigma = (f_1, f_2, \dots, f_n) \in N(D)_n$ and $ \beta \in \G^X_n(\sigma)$ \emph{i.e.} $\beta = (\ud{\beta},\beta)$. This pair consists of $(\ud{\beta_{n-1}}, \beta_{n-1})=\ud{\beta} \in \G^X_{n-1}(d_n(\sigma))$ and $\beta \in X(d_{n+1})_n$, where $f_n:d_n \to d_{n+1}$. The $n$-simplex $\delta$ satisfies the following two conditions:
	\begin{enumerate}
		\item $X(f_n)(p_2((\ud{\beta}))) = d_n(\beta)$.
		\item For $0 \le i \le n-2$
		\[
		(d_i(\ud{\beta}),d_i(\beta)) \in \G^X_{n-1}(d_i(\sigma)).
		\]
	\end{enumerate}
\end{rem}
\begin{rem}
	\label{zero-degree-comp}
	Let $\beta = (\ud{\beta},\beta) \in \G^X_n(\sigma)$, where $\sigma \in N(D)_n$ as in remark \eqref{desc-n-simp-Gr}. We observe that $d_n(\beta) = \ud{\beta} = (\ud{\beta_{n-1}},\beta_{n-1})$. Further, $d_{n-1}(\ud{\beta}) = \ud{\beta_{n-1}} = (\ud{\beta_{n-2}},\beta_{n-2}) \in \G^X_{n-2}(d_{n-2}(\sigma))$. Since $n$ is finite, there exists a $\beta_0 \in \G^X_{0}(d_1)$ such that
	\begin{equation}
	\label{beta-0}
	\beta_0 = d_1 \circ \cdots \circ d_{n-1} \circ d_n(\beta).
	\end{equation}
\end{rem}


The notion of \emph{relative nerve} was introduced in 
\cite[3.2.5.2]{JL}. Next we will review this notion:
\begin{df}
	\label{rel-Ner}
	  Let $D$ be a category, and $f : D \to \sSets$ a functor. The nerve of $D$ relative to $f$ is the simplicial set $N_f (D)$ whose $n$-simplices are sets consisting of:
	\begin{enumerate}
		\item[(i)] a functor $d:[n] \to D$; We write $d(i, j)$ for the image of $i \le j$ in $[n]$.
		\item[(ii)] for every nonempty subposet $J \subseteq [n]$ with maximal element $j$, a map $\tau^J:\Delta^J \to f(d(j))$,
		\item[(iii)] such that for nonempty subsets $I \subseteq J \subseteq [n]$ with respective maximal elements  $i \le j$, the following diagram commutes:
		\[
		\xymatrix{
			\Delta^I\ar[r]^{\tau^I} \ar@{_{(}->}[d] & f(d(i)) \ar[d]^{f(d(i, j))} \\
			\Delta^J\ar[r]_{\tau^J}  & f(d(j))
		}
		\]
	\end{enumerate}
\end{df}

For any $f$, there is a canonical map $p_f: N_f (D) \to N(D)$ down to the ordinary nerve of $D$, induced by the unique map to the terminal object $\Delta^0 \in \sSets$ \cite[ 3.2.5.4]{JL}. When $f$ takes values in quasi-categories, this canonical map is a coCartesian fibration.
\begin{rem}
	\label{Rel-Ner-edge}
	A vertex of the simplicial set $N_f(D)$ is a pair $(c, g)$, where $c \in Ob(D)$ and $g \in f(c)_0$.
	An edge $\ud{e}:(c, g) \to (d, k)$ of the simplicial set $N_f(D)$ consists of a pair $(e, h)$, where $e:c \to d$ is an arrow in $D$ and $h:f(e)_0(g) \to k$ is an edge of $f(d)$.
\end{rem}
An immidiate consequence of the above definition is the following proposition:
\begin{prop}
	\label{Rel-Ner-isom-func-val}
	Let $f:D \to \sSets$ be a functor, then the fiber of $p_f:N_f(D) \to N(D)$ over any $d \in Ob(D)$ is isomorphic to the simplicial set $f(d)$.
\end{prop}

The following lemma is a consequence of this definition and the above discussion:

\begin{lem}
	\label{isom-Rel-Ner}
	For each functor $X:D \to \sSets$, we have the following isomorphism in the category $\ovCatGen{\sSets}{N(D)}$:
	\[
	\rNGen{X}{D} \cong N_X(D).
	\]
	\end{lem}
\begin{proof}
	An $n$-simplex in $\rNGen{X}{D}$ is a pair $(\sigma, \beta)$, where $\sigma \in N(D)_n$. This $n$-simplex $\sigma$ can be viewed as a functor $\sigma:[n] \to D$. The inclusion of each non-empty subposet $i_J:J \subseteq [n]$ gives a map
	\[
	\left(\rNGen{X}{D} \right)(i_J):\left(\rNGen{X}{D} \right)_n \to \left(\rNGen{X}{D} \right)_J.
	\]
	We are using the fact that $J$ is isomorphic to an object of $\Delta$ which we also denote by $J$. The inclusion map can now be seen as a map in $\Delta$.
	This map gives us a $J$-simplex $\left(\rNGen{X}{D} \right)(i_J)((\sigma, \beta))$. Now the second projection map
	$p_2(\left(\rNGen{X}{D} \right)(i_J)((\sigma, \beta)))$ gives us a simplicial map:
	\[
	\Delta[J] \to X(\sigma(j'))
	\]
	where $j'$ is the maximal element of $J$. For an inclusion $J' \subseteq J$, condition $(iii)$ of definition \ref{rel-Ner} is satisfied because the composite $J' \subseteq J \subseteq [n]$ determines a composite map in $\Delta$. This defines a map $f:\left(\rNGen{X}{D} \right)_n \to N_X(D)_n$.
	Now we define the inverse map. An $n$-simplex $\gamma$ in $N_X(D)$ contains a functor $d:[n] \to D$ which uniquely determines an $n$-simplex $\sigma$ of $N(D)$. We recall that an $n$-simplex in $\rNGen{X}{D} $ is a pair $(\sigma, \beta)$ whose second component $\beta$ is a pair $(\ud{\beta}, \beta)$, where $\beta \in X(d(n))$. The $n$-simplex $\gamma$ contains a simplicial map
	\[
	\Delta[n] \to X(d(n)).
	\]
	We define $\beta$ to be the $n$-simplex of $X(d(n))$ which represents the above map. We have an inclusion $[n-1] \hookrightarrow [n]$ in $\Delta$. The $n$-simplex $\gamma$ contains another simplicial map
	\[
	\Delta[n-1] \to X(d(n-1)).
	\]
	We define $\ud{\beta}$ to be the pair $(\alpha, (\ud{\beta_{n-1}}, \beta_{n-1}))$ consisting of an $(n-1)$-simplex $\beta_{n-1}$ of $X(d(n-1))$ which represents the above map and $\alpha = d_n(\sigma)$. The first condition of remark \ref{desc-n-simp-Gr} is equivalent to the commutativity of the following diagram:
	\begin{equation*}
	\xymatrix{
	\Delta[n-1] \ar[r]^{\beta_{n-1}} \ar[d] & X(d(n-1)) \ar[d]^{X(d({n-1, n}))} \\
	\Delta[n] \ar[r]_{\beta}  & X(d(n))
    }
	\end{equation*}
The second condition of remark \ref{desc-n-simp-Gr} follows from definition \ref{rel-Ner} $(iii)$.
	\end{proof}

Next we will define a function object for the category $\sSetsGen{D}$.
We shall denote by $\mapGen{X}{Y}{D}$ the simplicial set of maps from $X$ to $Y$ in $\sSetsGen{D}$. An $n$-simplex in $\mapGen{X}{Y}{D}$ is a map $\Delta[n] \times X \to Y$ in $\sSetsGen{D}$, where $\Delta[n] \times (X, p) = (\Delta[n] \times X, pp_2)$, where $p_2$ is the projection $\Delta[n] \times X \to X$. The enriched category $\sSetsGen{D}$ admits tensor and cotensor products. The \emph{tensor product} of an object $X = (X, p)$ in $\sSetsGen{D}$ with a simplicial set $A$ is the objects
\[
A \times X = (A \times X, pp_2).
\]
The \emph{cotensor product} of $X$ by $A$ is an object of $\sSetsGen{D}$ denoted $\ExpG{A}{X}$. If $q:\ExpG{A}{X} \to \nGop$ is the structure map, then a simplex $x:\Delta[n] \to \ExpG{A}{X}$ over a simplex $y = qx:\Delta[n] \to N(D)$ is a map $A\times (\Delta[n], y) \to (X, p)$. The object $(\ExpG{A}{X}, q)$ can be constructed by the following pullback square in $\sSets$:
\begin{equation}
\label{co-tens-prod}
\xymatrix{
\ExpG{A}{X} \ar[r] \ar[d]_q & \mapGen{A}{X}{} \ar[d]^{\mapGen{A}{p}{}} \\
N(D) \ar[r] & \mapGen{A}{D}{}
}
\end{equation}
where the bottom map is the diagonal. There are canonical isomorphisms:
\begin{equation}
\mapGen{A \times X}{Y}{D} \cong \mapGen{A}{\mapGen{X}{Y}{D}}{} \cong \mapGen{X}{\ExpG{A}{Y}}{D}
\end{equation}
We now define a functor $\rNGenR{D}:\sSetsGen{D} \to [D, \sSets]$. For each $Y \in \sSetsGen{N(D)}$, the functor $\rNGenR{D}(Y)$ is defined as follows:
\begin{equation}
\label{right-adj-Gr}
\rNGenR{D}(Y)(d) := [N(\ovCatGen{d}{D}), Y]_D
\end{equation}
The contravariant functor $N(\ovCatGen{-}{D})$, see \eqref{rep-as-fib}, ensures that this defines a functor $\rNGenR{D}(Y):D \to \sSets$.

\begin{nota}
	For a simplicial map $p:X \to B$, we denote the \emph{fiber} of $p$ over an $n$-simplex $\sigma \in B_n$ by $X(\sigma)$. In other words, the simplicial set $X(\sigma)$ is defined by the following pullback square:
	\[
	\xymatrix{
		X(\sigma) \ar[r] \ar[d] & X \ar[d]^p \\
		\Delta[n] \ar[r]_\sigma & B 
	}
	\]
\end{nota}

 \section{Rectification of coCartesian fibrations}
\label{coCart-mdl-str}
In this section we will prove a \emph{rectification theorem} for \emph{coCartesian} fibrations of simplicial sets over the nerve of a small category $D$. It was shown in \cite[Ch. 3]{JL} that a coCartesian fibration over the nerve of a category is classified by a homotopy coherent diagram taking values in a (higher category of) quasi-categories. This result first appeared in \cite[]{JL}, wherein a version for marked simplicial sets of the higher Grothendieck construction functor constructed in the previous section is a right Quillen functor of a Quillen equivalence between the \emph{coCartesian} model category $\sSetsMGen{D}$ and the projective model category $[D, \sSetsMQ]$. However, in order to achieve our goal of constructing a homotopy colimit functor, we desire an \emph{approximation} of the higher Grothendieck construction functor whose derived functor is a left adjoint. More precisely, we will construct a left Quillen functor
\begin{equation*}
\gCLM{D}:\sSetsMQ \to [D, \sSetsMQ]
\end{equation*}
and show that its left derived functor is isomorphic to the right derived functor of $\mRNGen{-}{D}$. In order to establish this isomorphism we will first show that $\gCLM{D}$ is a left Quillen functor of a Quillen equivalence which we regard as the main result of this section. 
%
%
%
We begin with a review of coCartesin fibrations over the simplicial set $N(D)$. We will also review a model category structure on the category $\sSetsMGen{D}$ in which the fibrant objects are (essentially) coCartesian fibrations.
\begin{df}
	\label{p-CC-edge}
	Let $p:X \to S$ be an inner fibration of simplicial sets. Let $f:x \to y \in (X)_1$ be an edge in $X$. We say that $f$ is $p$-coCartesian if, for all $n \ge 2$ and every (outer) commutative diagram, there exists a (dotted) lifting arrow which makes the entire diagram commutative:
	\begin{equation}
	\xymatrix{
	\Delta^{\lbrace 0, 1 \rbrace} \ar@{_{(}->}[d] \ar[rd]^f \\
	\Lambda^0[n] \ar@{_{(}->}[d] \ar[r] & X \ar[d]^p \\
	\Delta[n] \ar[r] \ar@{-->}[ru] & S
   }
	\end{equation}
	\end{df}
\begin{rem}
	Let $M$ be a (ordinary) category equipped with a functor $p:M \to I$, then an arrow $f$ in $M$, which maps isomorphically to $I$, is coCartesian in the usual sense if and only if $f$  is $N(p)$-coCartesian in the sense of the above definition, where $N(p):N(M) \to \Delta[1]$ represents the nerve of $p$.
	\end{rem}
This definition leads us to the notion of a coCartesian fibration of simplicial sets:
\begin{df}
	\label{CC-fib}
	A map of simplicial sets $p:X \to S$ is called a \emph{coCartesian} fibration if it satisfies the following conditions:
	\begin{enumerate}
		\item $p$ is an inner fibration of simplicial sets.
		\item for each edge $p:x \to y$ of $S$ and each vertex $\ud{x}$ of $X$ with $p(\ud{x}) = x$, there exists a $p$-coCartesian edge $\ud{f}:\ud{x} \to \ud{y}$ with $p(\ud{f}) = f$.
		\end{enumerate}
	\end{df}
 A coCartesin fibration roughly means that it is upto weak-equivalence determined by a \emph{functor} from $S$ to a suitably defines $\infty$-category of $\infty$-categories. This idea is explored in detail in \cite[Ch. 3]{JL}.

\begin{nota}
	To each coCartesian fibration $p:X \to N(D)$ we can associate a marked simplicial set denoted $\Nt{X}$ which is composed of the pair $(X, \E)$, where $\E$ is the set of $p$-coCartesian edges of $X$
	\end{nota}
\begin{nota}
	\begin{sloppypar}
Let $(X, p), (Y, q)$ be two objects in $\sSetsMGen{D}$. We denote by  $[X, Y]_D^+$, the full (marked) simplicial subset of $[X, Y]^+$ spanned by maps in  $\sSetsMGen{D}(X, Y)$, namely spanned by maps in $[X, Y]^+$ which are compatible with the projections $p$ and $q$. We denote by  $\Flmap{X}{Y}{D}$, the full simplicial subset of $\Fl{[X, Y]}$ spanned by maps in  $\sSetsMGen{D}(X, Y)$. We denote by $\Shmap{X}{Y}{D} \subseteq \Shmap{X}{Y}{}$ the simplicial subsets spanned by maps in $\sSetsMGen{D}$.
\end{sloppypar}
\begin{df}
	\label{CC-Eq}
	A morphism $F:X \to Y$ in the category $\sSetsMGen{D}$ is called a \emph{coCartesian}-equivalence if for each coCartesian fibration $p:Z \to N(D)$, the induced simplicial map
	\[
	\Flmap{F}{\Nt{Z}}{D}:\Flmap{Y}{\Nt{Z}}{D} \to \Flmap{X}{\Nt{Z}}{D}
	\]
	is a categorical equivalence of simplicial-sets(quasi-categories).
	\end{df}
\begin{prop}
	\label{char-cc-eq}
	Let $u:X \to Y$ be a map in $\sSetsMGen{D}$, then the following are equivalent
	\begin{enumerate}
	 \item $u$ is a coCartesian equivalence.
	 \item For each functor $Z:D \to \sSetsM$, such that $Z(d)$ is a quasi-category whose marked edges are equivalences, the following (simplicial) map is a categorical equivalence:
	\[
	\Flmap{u}{ \mRNGen{Z}{D}}{D}:\Flmap{Y}{\mRNGen{Z}{D}}{D} \to \Flmap{X}{\rNGen{Z}{D}}{D}
	\]
	\item For each functor $Z:D \to \sSetsM$, such that $Z(d)$ is a quasi-category whose marked edges are equivalences, the following map is a bijection:
	\[
	\pi_0\Shmap{u}{ \mRNGen{Z}{D}}{D}:\pi_0\Shmap{Y}{\mRNGen{Z}{D}}{D} \to \pi_0\Shmap{X}{\rNGen{Z}{D}}{D}
	\]
	\end{enumerate}
\end{prop}
\begin{proof}
	$(1 \Rightarrow 2)$ Follows from the definition of coCartesian equivalence because $\mRNGen{Z}{D}$ is a coCartesian fibration under the given hypothesis.

	 Let us assume that $\Flmap{u}{ \mRNGen{Z}{D}}{D}$ is a categorical equivalence of quasi-categories for each functor $Z$ satisfying the given hypothesis.  This imples that $\Flmap{u}{ \Nt{T}}{D}$ is a categorical equivalence if and only if $\Flmap{u}{ \mRNGen{Z(T)}{D}}{D}$ is one.
		
		 $(2 \Rightarrow 3)$ We recall from \cite[Prop. 3.1.3.3]{JL} and \cite[Prop. 3.1.4.1]{JL} that, for any coCartesian fibration $\Nt{T} \in \sSetsMGen{D}$, the simplicial map $\Flmap{u}{ \Nt{T}}{D}$ is a categorical equivalence if and only if the map $\Shmap{u}{\Nt{T}}{D}$ is a homotopy equivalence of Kan complexes. This implies that $\pi_0\Shmap{u}{ \mRNGen{Z}{D}}{D}$ is a bijection.
		 
		  $(3 \Rightarrow 1)$ 
		  We recall from \cite[Cor. 3.1.4.4]{JL} that the coCartesian model category is a simplicial model category with simplicial function object given by the bifunctor $\Shmap{-}{-}{D}$.
		  This implies that $u$ is a coCartesian equivalence if and only if $\pi_0\Shmap{u}{\Nt{W}}{D}$ is a bijection for each fibrant object $W$ of the coCartesian model category. By  \cite[Prop. 3.1.4.1]{JL} we may replace $W$ by a coCartesian fibration $W \cong \Nt{T}$.
		  Further, it follows from \cite[Prop. 3.2.5.18(2)]{JL}  that for each cocartesian fibration $\Nt{T}$ there exists a functor $Z(T):D \to \sSetsM$, which satisfies the assumptions of the functor in the statement of the proposition, such that there is map
		  	\[
		  	F_T:\Nt{T} \to \mRNGen{Z(T)}{D}
		  	\]
		  	which is a coCartesian equivalence. Now it follows that $u$ is a coCartesian equivalence if and only if $\pi_0\Shmap{u}{\mRNGen{Z(T)}{D}}{D}$ is a bijection for each functor $Z$ satisfying the conditions mentioned in the statement of the proposition.
		 
\end{proof}

\end{nota}
Next we will recall a model category structure on the overcategory $\sSetsMGen{D}$ from \cite[Prop. 3.1.3.7.]{JL} in which fibrant objects are (essentially) coCartesian fibrations.
\begin{thm}
	\label{CC-Mdl-Str}
	There is a left-proper, combinatorial model category structure on the category $\sSetsMGen{D}$ in which a morphism is 
	\begin{enumerate}
		\item a cofibration if it is a monomorphism when regarded as a map of simplicial sets.
		\item a weak-equivalences if it is a coCartesian equivalence.
		\item a fibration if it has the right lifting property with respect to all maps which are simultaneously cofibrations and weak-equivalences. 
		\end{enumerate}
	\end{thm}

We have defined a function object for the category $\sSetsMGen{D}$ above.
The simplicial set $\Flmap{X}{Y}{D}$ has verices, all maps from $X$ to $Y$ in $\sSetsMGen{D}$. An $n$-simplex in $\Flmap{X}{Y}{D}$ is a map $\Fl{\Delta[n]} \times X \to Y$ in $\sSetsMGen{D}$, where $\Fl{\Delta[n]} \times (X, p) = (\Fl{\Delta[n]} \times X, pp_2)$, where $p_2$ is the projection $\Fl{\Delta[n]} \times X \to X$. The enriched category $\sSetsMGen{D}$ admits tensor and cotensor products. The \emph{tensor product} of an object $X = (X, p)$ in $\sSetsMGen{D}$ with a simplicial set $A$ is the objects
\[
\Fl{A} \times X = (\Fl{A} \times X, pp_2).
\]
The \emph{cotensor product} of $X$ by $A$ is an object of $\sSetsMGen{D}$ denoted $\expG{A}{X}$. If $q:\expG{A}{X} \to \Sh{N(D)}$ is the structure map, then a simplex $x:\Fl{\Delta[n]} \to \expG{A}{X}$ over a simplex $y = qx:\Delta[n] \to \Sh{N(D)}$ is a map $\Fl{A} \times (\Fl{\Delta[n]}, y) \to (X, p)$. The object $(\expG{A}{X}, q)$ can be constructed by the following pullback square in $\sSetsM$:
\begin{equation*}
\xymatrix{
\expG{A}{X} \ar[r] \ar[d]_q & \mapMS{\Fl{A}}{X} \ar[d]^{\mapMS{\Fl{A}}{p}} \\
\Sh{N(D)} \ar[r] & \mapMS{\Fl{A}}{\Sh{N(D)}}
}
\end{equation*}
where the bottom map is the diagonal. There are canonical isomorphisms:
\begin{equation}
\Flmap{\Fl{A} \times X}{Y}{D} \cong \left[A, \Flmap{X}{Y}{D} \right] \cong \Flmap{X}{\expG{A}{Y}}{D}
\end{equation}
\begin{rem}
	\label{Simp-Mdl-Cat}
	The coCartesian model category structure on $\sSetsMGen{D}$ is a simplicial model category structure with the simplicial Hom functor:
	\[
	\Shmap{-}{-}{D}:\sSetsMGen{D}^{op} \times \sSetsMGen{D} \to \sSets.
	\]
	This is proved in \cite[Corollary 3.1.4.4.]{JL}. The coCartesian model category structure is a $\sSetsQ$-model category structure with the function object given by:
	\[
	\Flmap{-}{-}{D}:\sSetsMGen{D}^{op} \times \sSetsMGen{D} \to \sSets.
	\]
	This is remark \cite[3.1.4.5.]{JL}.
	\end{rem}
\begin{rem}
	\label{enrich-mar-sSets}
	The coCartesian model category is a $\sSetsMQ$-model category with the Hom functor:
	\[
	\mapGenM{-}{-}{D}:\sSetsMGen{D}^{op} \times \sSetsMGen{D} \to \sSetsM.
	\]
	This follows from \cite[Corollary 3.1.4.3]{JL} by taking $S = N(D)$ and $T = \Delta[0]$, where $S$ and $T$ are specified in the statement of the corallary.
	\end{rem}
 \begin{df}
 	\label{mar-rel-Ner}
 	Let $F:D \to \sSetsM$ be a functor. We can compose it with the forgetful functor $U$ to obtain a composite functor $F:D \overset{F} \to \sSetsM \overset{U} \to \sSets$. The \emph{marked} Grothendieck construction of $F$, denoted $\mRNGen{F}{D}$, is the marked simplicial set $\left(\mRNGen{F}{D}, \E \right)$,
 	where the set $\E$ consists of those edges $\ud{e} = (e, h)$ of $\mRNGen{F}{D}$, see remark \ref{Rel-Ner-edge}, which determines a marked edge of the marked simplicial set $F(d)$, where $e:c \to d$ is an arrow in $D$.
 	\end{df}
 The above construction of the marked Grothendieck construction determines a functor 
 \begin{equation}
 \label{mark-Gr-const}
 \mRNGen{-}{D}:[D, \sSetsM] \to \sSetsMGen{D}. 
 \end{equation}
 The functor $ \mRNGen{-}{D}$ has a left adjoint which we denote by $\mRNGenL{\bullet}{D}$, see \cite[]{JL}. This functor is defined on objects as follows:
 \begin{equation*}
 \mRNGenL{X}{D}(d) = X \underset{\Sh{N(D)}} \times \Sh{N(D/d)},
 \end{equation*}
 where $(X, p)$ is an object in $\sSetsMGen{D}$.
 \begin{nota}
 	We will sometimes denote $\mRNGenL{X}{D}$ by $\mRNGenL{\bullet}{D}(X)$.
 	\end{nota}
 \begin{rem}
 	\label{left-adj-pres-fib-obj}
 	If $(X, p)$ is a fibrant object in $\sSetsMGen{D}$ then $\mRNGenL{X}{D}$ is a fibrant object in the projective model category $[D, \sSetsMQ]$.
 	\end{rem}

 Next we will define a marked version of the functor $\rNGenR{D}$, denoted $\mRNGenR{D}$:
 \begin{equation*}
 \mRNGenR{D}(X)(d) := [\Sh{N(\ovCatGen{d}{D})}, X]^+_D
 \end{equation*}
 where $X$ is an object of $\sSetsMGen{D}$. This functor has a left adjoint which we denote by $\gCLM{D}$.
 
 \begin{df}
 	\label{unit-comp-Gr}
 	Let $X:D \to \sSetsM$ be a functor.
 	For each $d \in D$ we define a map of marked simplicial sets 
 	\[\eta^+_X(d):X(d) \to [N(\ovCatGen{d}{D}), \mRNGen{X}{D}]^+_D.
 	\]
 	Let $x \in X(d)_n$ be an $n$-simplex in $X(d)$. This $n$-simplex defines a canonical map $\eta^+_X(d)(x):N(\ovCatGen{d}{D}) \times \Delta[n] \to \mRNGen{X}{D}$ in $\sSetsMGen{D}$ whose value on $(\cn{n}, id_n) \in (N(\ovCatGen{d}{D}) \times \Delta[n])_n$ is the image of $x$ in $\mRNGen{X}{D}$, namely the $n$-simplex $(\ud{x}, x)$, where $\ud{x} = (\ud{x_{n-1}}, d_n(x))$. We recall that a $k$-simplex in $\Delta[n]$ is a map $\alpha:[k] \to [n]$ in the category $\Delta$ and therefore it can be written as $\Delta[n](\alpha)(id_n)$.
 	For a $k$-simplex $((g, f_1, \dots, f_{k+1}), \alpha)$ in $N(\ovCatGen{d}{D}) \times \Delta[n]$, we define
 	\[
 	\eta^+_X(d)(x)((g, f_1, f_2, \dots, f_{k+1}), \alpha) := X(f_{k+1} \circ f_k \circ \cdots \circ g)(X(d)(\alpha)(x)).
 	\]
 	This defines the (simplicial) map $\eta^+_X(d)(x)$. These simplicial maps glue together into a natural transformation $\eta^+_X$.
 	\end{df}
%
 Now we define a map $\iota_d^+$ in $\sSetsMGen{D}$:
 \begin{equation}
 \label{main-gen-local-mar}
 \xymatrix{
 	\Fl{\Delta[0]} \ar[rr]^{id_d} \ar[rd]_d  && \Sh{N(\ovCatGen{d}{D})} \ar[ld] \\
 	&\Sh{N(D)}
 }
 \end{equation}
 \begin{lem}
 	\label{main-lemma-mar}
 	For each $d \in D$ the morphism $\iota^+_d$ defined in \eqref{main-gen-local-mar} is a coCartesian equivalence.
 \end{lem}
 \begin{proof}
 	We will show that for each functor $Z:D \to \sSetsM$ such that, for each $d \in D$, $Z(d)$ is a quasi-category whose marked edges are equivalences, we have the following bijection:
 	\[
 	\pi_0\Shmap{\iota^+_d}{\rNGen{Z}{D}}{D}:\pi_0\Shmap{N(\ovCatGen{d}{D})}{\rNGen{Z}{D}}{D}  \to \pi_0\Shmap{\Delta[0]} {\rNGen{Z}{D}}{D} \cong \pi_0(J(Z(d))),
 	\]
 	where $J(Z(d))$ is the largest Kan complex contained in $Z(d)$.
 	Let $z \in J(Z(d))_0$ be a vertex of $\Sh{J(Z(d))}$. We will construct a morphism $F_z:N(\ovCatGen{d}{D}) \to \rNGen{Z}{D}$ in the category $\sSetsMGen{D}$.  The vertex $z$ represents a natural transformation
 	\[
 	T_z:D(d, -) \Rightarrow Z
 	\]
 	such that $T_z(id_d) = z$. Since $N(\ovCatGen{d}{D}) \cong \rNGen{D(d, -)}{D}$ therefore we have a map
 	\[
 	F_z:N(\ovCatGen{d}{D}) \cong \rNGen{D(d, -)}{D} \overset{\rNGen{T_z}{D}} \to \rNGen{Z}{D}
 	\]
 	in $\sSetsMGen{D}$ such that $F_z(id_d) = z$. Thus we have shown that the map $\pi_0\Shmap{\iota^+_d} {\rNGen{Z}{D}}{D}$ is a surjection.
 	
 	Let $f:y \to z$ be an edge of $J(Z(d))$, then by the (enriched) Yoneda's lemma followed by an application of the Grothendieck construction functor, this edge uniquely determines a map
 	\[
 	T_f:N(\ovCatGen{d}{D}) \times \Delta[1] \to \rNGen{Z}{D}
 	\]
 	in $\sSetsMGen{D}$ such that $F_z((id_d, id_1)) = f$. Thus we have shown that the map $\pi_0\Shmap{\iota^+_d} {\rNGen{Z}{D}}{D}$ is also an injection.
 	
 \end{proof}

\begin{lem}
	\label{unit-Eq-mar}
	For any projectively fibrant functor $X:D \to \sSetsM$, the  map $\eta^+_X$ defined in \ref{unit-comp-Gr} is an objectwise categorical equivalence of marked simplicial sets.
\end{lem}
\begin{proof}
	Under the hypothesis of the lemma, it follows from \cite[Prop. 3.2.5.18(2)]{JL} and lemma \ref{isom-Rel-Ner} that $\mRNGen{X}{D}$ is a fibrant object in the coCartesian model category. Now lemma \ref{main-lemma-mar} and remark \ref{enrich-mar-sSets} gives us, for each $d \in D$, the following homotopy equivalence in $\sSetsMQ$:
	\[
	[\iota^+_d, \rNGen{X}{D}]^+_D:[\Sh{N(\ovCatGen{d}{D})}, \mRNGen{X}{D}]^+_D \to [\Fl{\Delta[0]}, \mRNGen{X}{D}]^+_D \cong X(d)
	\]
	such that $c \circ [\iota_d, \rNGen{X}{D}]_D \circ \eta_X(d) = id_{X(d)}$, where $c$ is the canonical isomorphism between the fiber of $p:\rNGen{X}{D} \to N(D)$ over $d \in D$ and $X(d)$ \emph{i.e.} the value of the functor $X$ on $d$.
	Now the $2$ out of $3$ property of weak equivalences in a model category tells us that $\eta_X(d)$ is a homotopy equivalence for each $d \in D$ therefore $\eta^+_X$ is an objectwise homotopy equivalence in $[D,\sSetsMQ]$.
\end{proof}
 
 An immediate consequence of the definition of the right adjoint functor $\mRNGenR{D}$ is the following lemma:
 \begin{lem}
 	\label{Qu-Eq-L-R-mar}
 	The adjunction $(\gCLM{D}, \mRNGenR{D})$ is a Quillen adjunction between the projective model category structure on $[D, \sSetsMQ]$ and coCartesian model category $\sSetsMGen{D}$.
 \end{lem}
 \begin{proof}
 	The coCartesian model category is a $\sSetsMQ$-model category, see remark \ref{enrich-mar-sSets}. This implies that $\mRNGenR{D}$ maps (acyclic) fibrations in the coCartesian model category to (acyclic) projective fibrations in $[D, \sSetsM]$ which are objectwise (acyclic) fibrations of marked simplicial sets.
 \end{proof}
 
 Now we get to the main result of this paper.
The natural isomorphism constructed in the proof of the above theorem implies the following proposition:
\begin{thm}
	\label{nat-isom-left-right-der-func}
	The (total) left derived functor of the left Quillen functor $\pFGen{\bullet}{D}$ is naturally isomorphic to the (total) right derived functor of $\mRNGenR{D}$.
	\end{thm}
\begin{proof}
	Let $(Q, q)$ be the cofibrant replacement functor and $(R, r)$ be the fibrant replacement functorin the coCartesian model category, both obtained by the chosen factorization system on the coCartesian model category. Let $Z$ be an object in $\sSetsMGen{D}$, we have the following chain of composable arrows in the projective model category $[D, \sSetsMQ]$:
		\begin{multline*}
	\pFGen{\bullet}{D}\left( Q(Z) \right) \overset{K} \to \pFGen{\bullet}{D}\left( Q\left( \mRNGen{\pFGen{\bullet}{D}(R(Z))}{D} \right)\right)
	\overset{q} \to \\
	\pFGen{\bullet}{D}\left( \mRNGen{\pFGen{\bullet}{D}(R(Z))}{D} \right) 
	\overset{L} \to \pFGen{\bullet}{D}(R(Z)) \overset{M} \to \mRNGenR{D}\left( \mRNGen{\pFGen{\bullet}{D}(R(Z))}{D}  \right)
	\end{multline*}
	The map $L$ in the above composite is $\pFGen{\bullet}{D}(\epsilon_{R(Z)})$ where $\epsilon$ is the counit of the adjunction $(\pFGen{\bullet}{D}, \mRNGen{-}{D})$. It follows from \cite[Prop. 3.2.5.18]{JL} that $\epsilon_{R(Z)}$ is a coCartesian equivalence and the left Quillen functor $\pFGen{\bullet}{D}$ preserves all weak-equivalences therefore $L$ is a weak-equivalence. The map $K$ in the above composite is the following:
	\begin{multline*}
	\pFGen{\bullet}{D}(Q(Z)) \overset{\pFGen{\bullet}{D}Q(\eta_{Z})} \to \pFGen{\bullet}{D} \left(Q\left(  \mRNGen{\pFGen{\bullet}{D}(Z)}{D} \right)  \right)
	\overset{\pFGen{\bullet}{D} \left(Q\left(  \mRNGen{\pFGen{\bullet}{D}(r_Z)}{D} \right)  \right)} \to \\
	 \pFGen{\bullet}{D} \left(Q\left(  \mRNGen{\pFGen{\bullet}{D}(R(Z))}{D} \right)  \right)
	\end{multline*}
	The unit natural transformation $\eta$ provides us the following commutative diagram:
	\[
	\xymatrix{
	Z \ar[r]^{\eta_Z \ \ \ \ \ \ } \ar[d]_{r_Z } & \mRNGen{\pFGen{\bullet}{D}(Z)}{D} \ar[d]^{\mRNGen{\pFGen{\bullet}{D}(r_Z)}{D}} \\
	R(Z) \ar[r]_{\eta_{R(Z)} \ \ \ \ \ \ \ \ \ \ \  }  & \mRNGen{\pFGen{\bullet}{D}(R(Z))}{D} 
    }
	\]
	The left downward arrow is clearly a weak equivalence. It follows from remark \ref{left-adj-pres-fib-obj} and \cite[Prop. 1.3.13]{Hovey} that the bottom right arrow is also a weak-equivalence in the above commutative square. This observation along with the facts that both functors $Q$ and $\pFGen{\bullet}{D}$ preserve weak-equivalences prove that $K$ is a weak-equivalence. The morphism $M$ is the unit map $\eta^+_{\pFGen{\bullet}{D}(R(Z))}$ which is a weak-equivalence by lemma \ref{unit-Eq-mar}. Thus we have constructed a natural weak-equivalence between the following two composite functors:
	\[
	\sSetsMGen{D} \overset{Q} \to (\sSetsMGen{D})^c \overset{\pFGen{\bullet}{D}} \to [D, \sSetsMQ]
	\]
	which we denote by $F$, and
	\begin{multline*}
	\sSetsMGen{D} \overset{R} \to (\sSetsMGen{D})^f \overset{\pFGen{\bullet}{D}} \to [D, \sSetsMQ] \overset{\mRNGen{-}{D}} \to \sSetsMGen{D} \overset{\mRNGenR{D}} \to [D, \sSetsMQ]
	\end{multline*}
	which we denote by $G$. We denote the above natural weak-equivalence by $\tau:F \Rightarrow G$. We observe that the  functor induced on the homotopy categories $Ho(F)$ is the same as the (total) left derived functor of $\pFGen{\bullet}{D}$. The composite functor $G$ is also a weak-equivalence preserving functor therefore it induces a functor on the homotopy categories which we denote by $Ho(G)$. It follows from \cite[lemma 1.2.2(iii)]{Hovey} that there is a unique natural isomorphism $Ho(\tau):Ho(F) \Rightarrow Ho(G)$. The Quillen equivalence \cite[]{JL}
   implies that $Ho(F)$ is an equivalence of categories therefore $Ho(G)$ is also an equivalence by the natural isomorphism $Ho(\tau)$.
	
	We construct another natural weak equivalence from the composite functor
	\[
	\sSetsMGen{D} \overset{R} \to (\sSetsMGen{D})^f \overset{\mRNGenR{D}} \to [D, \sSetsMQ],
	\]
	denoted by $H$, to the functor $G$, denoted by $\beta:H \Rightarrow G$. For any object $Z \in \sSetsMGen{D}$, the morphism $\beta_Z$ is defined as follows:
	\[
	\mRNGenR{D}\left( R(Z)  \right) \overset{\mRNGenR{D}\left( \eta_{R(Z)}  \right)} \to \mRNGenR{D}\left( \mRNGen{\pFGen{\bullet}{D}(R(Z))}{D}  \right)
	\]
	The map $\beta_Z$ is a weak-equivalence for each $Z$. We observe that the functor $H$ preserves weak-equivalences and therefore induces a functor on the homotopy categories which we denote by $Ho(H)$. This functor is the same as the (total) right derived functor of the right Quillen functor $\mRNGenR{D}$. The same argument as above tells us that $\beta$ uniquely determines a natural isomorphism $Ho(\beta):Ho(H) \Rightarrow G$. 
	\end{proof}

The above theorem along with lemma \ref{Qu-Eq-L-R-mar} implies the following corollary:
\begin{coro}
	\label{main-res-CC-mar}
	The Quillen pair $(\gCLM{D}, \mRNGenR{D})$ is a Quillen equivalence.
\end{coro}
 We have shown the existence of the left Quillen functor $\gCLM{D}$ above. Now we will provide a construction of this left adjoint. Let $F:D \to \sSetsM$ be a functor. An $n$-simplex in the (total space of) $\gCLM{D}(F)$ is a pair $(\sigma, x)$, where $\sigma = (d_0 \overset{f_1} \to d_1 \overset{f_1} \to d_2 \overset{f_2} \to \cdots \overset{f_n} \to d_n)$ is an $n$-simplex of $N(D)$ and $x \in F(d_0)_n$. An edge $(f, x)$ is marked in $\gCLM{D}(F)$ if and only if the edge $x \in F(d_0)_1$ is a marked edge.
 
 \begin{rem}
 	\label{rel-bar-const}
 	The marked simplicial set (total space of) $\gCLM{F}$ is the bar construction $|B(\ast, C, F)|$ of the functor $F:D \to \sSetsM$.
 	We recall that $\sSetsMQ$ is a simplicial model category. See \cite{Shul}, \cite{Meyer} for the above notation and a description of a bar construction in a simplicial model category.
 	\end{rem}
 The following lemma is an easy consequence of the above description of the left Quillen functor $\gCLM{D}$:
 \begin{lem}
 	\label{fib-img-L}
 	For each functor $F:D \to \sSetsMQ$ which is a fibrant object in the projective model category $[D, \sSetsMQ]$, the object $\gCLM{D}(F)$ is a fibrant object of $\sSetsMGen{D}$.
 	\end{lem}
 
 Based on the above description of the functor $\gCLM{D}$, it is easy to see an inclusion natural transformation:
 \begin{equation*}
 \iota_D:\gCLM{D} \Rightarrow \mRNGen{-}{D}.
 \end{equation*}
 For a functor $F:D \to \sSetsM$ the simplicial map $\iota_D(F)$ determined by the above natural transformation can be described, in degree $n$, as follows:
 \[
 (\sigma, x) \mapsto (\sigma, (\ud{\beta}, F(f_n \circ \cdots \circ f_1)(x))),
 \]
 where $\ud{\beta} = (d_{n-1}(F(f_{n-2} \circ \cdots \circ f_1)(d_n(x))), d_n(F(f_{n-1} \circ \cdots \circ f_1)(x)))$.
 \begin{prop}
 	\label{comp-L-Gr-const}
 	For a fibrant diagram $F$ in $[D, \sSetsMQ]$, the map $\iota_D(F)$ is a coCartesian equivalence.
 	\end{prop}
 \begin{proof}
 	We observe that the map $\iota_D(F)$ induces an isomorphism on the fibers. Now the proposition follows from \cite[Prop. 3.1.3.5.]{JL}
 
 	\end{proof}
 \begin{coro}
 	\label{homt-func-L}
 	The left Quillen functor $\gCLM{D}$ is a homotopical functor.
 	\end{coro}
 \begin{proof}
 	Since $\mRNGen{-}{D}$ is a right Quillen functor, proposition \ref{comp-L-Gr-const} implies that $\gCLM{D}$ preserves weak-equivalences between fibrant objects. Let $H:X \to Y$ be a weak equivalence in the projective model category $[D, \sSetsMQ]$. The chosen fibrant replacement functor $(r, R)$ gives us the following commutative diagram in $[D, \sSetsMQ]$ wherein the horizontal maps are acyclic cofibrations:
 	\begin{equation*}
 	\xymatrix{
 	X \ar[d]_H \ar[r]^{r(X)} & R(X) \ar[d]^{R(H)} \\
 	Y \ar[r]_{r(Y)} & R(Y)
    }
 	\end{equation*}
 	Applying the left Quillen functor $\gCLM{D}$ to the above commutative square gives us the following commutative square in $\sSetsMGen{D}$:
 		\begin{equation*}
 	\xymatrix@C=16mm{
 		\gCLM{D}(X) \ar[d]_{\gCLM{D}(H)} \ar[r]^{\gCLM{D}(r(X))} & \gCLM{D}(R(X)) \ar[d]^{\gCLM{D}(R(H))} \\
 		\gCLM{D}(Y) \ar[r]_{\gCLM{D}(r(Y))} & \gCLM{D}(R(Y))
 	}
 	\end{equation*}
 	The two horizontal arrows are weak-equivalences because $\gCLM{D}$ preserves acyclic cofibrations. We have observed above that $\gCLM{D}$ preserves weak-equivalences between fibrant objects therefore $\gCLM{D}(R(H)$ is a weak-equivalence. Now the two out of three property of weak-equivalences in a model category implies that $\gCLM{D}(H)$ is a weak-equivalence.

 	\end{proof}
 For the terminal map $u:N(D) \to \ast$ induces a \emph{pullback} functor
 \[
 u^*:\sSetsM \to \sSetsMGen{D}
 \]
 In this situation $u^*(X) = (N(D) \times X, p_1)$ for each object $X \in \sSetsM$, where $p_1$ is the projection to $N(D)$. The functor $u^*$ has a left adjoint
 \[
 u_!:\sSetsMGen{D} \to \sSetsM
 \]
 which maps an object $(X, p) \in \sSetsMGen{D}$ to the compopsite $X \to N(D) \to \ast$. It is easy to check that the adjunction $(u_!, u^*)$ is a Quillen adjunction. Now we define a homotopy colimit functor, see definition \ref{hocolim-func}:
 \begin{prop}
 	\label{hocolim-func-mar-SS}
 	The composite right Quillen functor $u_! \circ \gCLM{D}$ is a homotopy colimit functor.
 	\end{prop}
 \begin{proof}
 	We have seen above (cor. \ref{homt-func-L}) that the functor $\gCLM{D}$ is homotopical. The functor $u_!$ is also homotopical because it is a left Quillen functor wherein every object of the domain model category is cofibrant.
 	
 	Let $(R, r)$ be the fibrant replacement functor determined by the choses functorial factorization on the model catgory $ \sSetsMQ$.
 	We now construct a natural weak equivalence
 	\begin{equation*}
 		\label{hocol-nat-WE}
 		\delta:\mRNGenR{D}u^*R \Rightarrow \Delta R
 		\end{equation*}
 	For a marked simplicial set $X$, the functor $\mRNGenR{D}u^*R(X)$ is defined as follows:
 	\[
 	\mRNGenR{D}u^*R(X)(d) = \mRNGenR{D}(\Sh{N(D)} \times R(X))(d) = \mapFl{\Sh{N(\ovCatGen{d}{D})}}{N(D) \times R(X)}_D
 	 \]
 	The functor $\Delta R$ is isomorphic to the functor 
 	\[
 	\mapFl{(\Sh{\Delta[0]}, -)}{N(D) \times R(-)}_D: \sSetsM \to [D, \sSetsM],
 	\]
 	where $(\Sh{\Delta[0]}, d)$ denotes the map $d:\Sh{\Delta[0]} \to \Sh{N(D)}$. Now the desired natural weak-equivalence $\delta$ is defined as follows for each pair $(X, d) \in Ob(\sSetsM) \times Ob(D)$:
 	\[
 	\delta_X(d) := \mapFl{\iota_d^+}{N(D) \times R(X)}_D:\mRNGenR{D}u^*R(X)(d) \to R(X),
 	\]
 	 where the map $\iota_d^+$ is defined in \eqref{main-gen-local-mar}. The natural transformation $\eta_X$ is a natural weak equivalence because $\iota_d^+$ is a weak equivalence.
 	 
 	 \begin{sloppypar}
 	  Now it follows from \cite[33.9 (ii)]{DHKS} that the natural weak-equivalence $\mapFl{\iota_d^+}{N(D) \times R(-)}_D$ induces a natural isomorphism between the (derived) functors $Ho(\mRNGenR{D}u^*R)$ and $Ho(\Delta R)$. Since the functor $\Delta$ is homotopical, it follows from \cite[33.9 (ii)]{DHKS} that there is a natural isomorphism between $Ho(\Delta R)$ and $Ho(\Delta)$.
 	\end{sloppypar}
 	\end{proof}
 
The proposition has the following corollary:
\begin{coro}
	\label{der-funct-L-der-GC-eq}
	The (total) left derived functor of $\gCLM{D}$ is naturally isomorphic to the (total) right derived functor of $\mRNGen{-}{D}$.
	\end{coro}
\begin{nota}
	The total (total) right derived functor of $\mRNGen{-}{D}$ refers to the total right derived functor of the relative nerve functor for marked simplicial sets, see \cite[Prop. 3.2.5.18(2)]{JL}.
	\end{nota}

\begin{prop}
	\label{simp-adj}
	For all functor $X:D \to \sSetsM$ and $K \in \sSets$ we have the following isomorphism
	\[
	\gCLM{D}(X \otimes K) \cong \gCLM{D}(X) \otimes K.
	\]
	\end{prop}

%
%
%
%
\section[A homotopy colimit functor]{A homotopy colimit functor for diagrams of quasi-categories}
\label{hoColim-QCat}

The homotopy colimit of a functor $F:D \to \sSetsK$ has a standard construction, namely it is the \emph{diagonal} of the bisimplicial set obtained by applying the nerve functor to the  transport category (Grothendieck construction) of the $n$-simplex functors $F_n$, see \cite[Ch. IV]{GJ}.
In this section we will present a homotopy colimit construction for functors taking values in the Joyal model category of simplicial sets $\sSetsQ$. Our construction is a modification of the aforementioned construction. Our construction will use the homotopy colimit functor constructed on the category of functors taking values in $\sSetsMQ$, see proposition \ref{hocolim-func-mar-SS}. We exhibit a very natural embedding of simplicial sets into marked simplicial sets by marking \emph{equivalences} in simplicial sets. We go on to show that this embedding is an equivalence of the underlying homotopy theories of the two model categories in context. Now the homotopy colimit of a functor $F:D \to \sSetsQ$ is obtained by first composing the functor with the aforementioned embedding $F:D \to \sSetsQ \to \sSetsMQ$, then applying the homotopy colimit functor $u_! \circ \gCLM{D}$ to $F$ and finally inverting all the coCartesian edges of $\gCLM{D}(F) $.

 We begin by recalling the notion of a homotopy colimit functor, see \cite{DS95}, \cite{DHKS} for more detail:
\begin{df}
	\label{hocolim-func}
	A \emph{homotopy colimit functor} on the functor category $[D, \M]$, where $\M$ is a model category, is a homotopical functor
	\[
	hocolim:[D, \M] \to \M
	\]
	such that its induced functor on the homotopy category $Ho(hocolim)$ is a left adjoint to the functor
	\[
	Ho(\Delta):Ho(\M) \to Ho[D, \M].
	\] 
	\end{df}
 In this paper we restrict to those model categories which induce a projective model category structure on the functor category $[D, \M]$.

We recall that an edge in a quasi-category $X$ is called an equivalence if it determines an isomorphism in the homotopy category $Ho(X)$.
We want to extend this definition to all simplicial-sets. We recall that the nerve functor $N:\sSets \to \Cat$ has a left adjoint which is denoted by $\tau_1$, the unit map of this adjunction is denoted by $\eta:id \Rightarrow N\tau_1$
\begin{df}
An edge $y$ of a simplicial set $S$ is called an \emph{equivalence} in the simplicial-set $S$ if its image $\eta_S(y)$ in the quasi-category $N\tau_1(S)$ is an equivalence.
\end{df}
\begin{rem}
Unwinding the definition of the functor $\tau_1$ we observe that an edge $y:a \to b$ in a simplicial set $S$ is an equivalence if and only if there is another edge $\inv{y}:b \to a$ in $S$ and a pair of $2$-simpleces $\sigma, \beta \in S_2$ such that $d_0(\sigma) = \inv{y}$, $d_0(\beta) = y$, $d_2(\sigma) = y$, $d_2(\beta) = \inv{y}$ and $d_1(\sigma) = \unit{a},  d_1(\beta) = \unit{b}$.
\end{rem}
\begin{prop}
	\label{Smap-Eq-to-Eq}
A morphism $F:S \to T$ of simplicial sets maps an equivalence in $S$ to an equivalence in $T$.
\end{prop}
\begin{proof}
Let $F:S \to T$ be a morphism of simplicial-sets and $y$ be an equivalence in $S$.
The unit map of the adjunction $(\tau_1, N)$ provides us with the following commutative square:
\[
\xymatrix{
S \ar[r]^F \ar[d]_{\eta_S} & T \ar[d]^{\eta_T} \\
N\tau_1(S) \ar[r]_{N\tau_1(F)} & N\tau_1(T)
}
\]
By assumption $\eta_S(y)$ is an equivalence in the quasi-category $N\tau_1(S)$. By the above commutative diagram, it is sufficient to show that $N\tau_1(F)(\eta_S(y))$ is an equivalence in the quasi-category $N\tau_1(T)$. The assumption that $\eta_S(y)$ is an equivalence in the quasi-category $N\tau_1(S)$ implies that $y$ is a representative of an isomorphism in $\tau_1(S)$ and the functor $\tau_1(F)$ maps this isomorphism to an isomorphism in $\tau_1(T)$ which determines an equivalence $N\tau_1(F)(\eta_S(y))$ in $N\tau_1(T)$.
\end{proof}

	There is an inclusion map $i:\Delta[1] \hookrightarrow J$, where $J$ is the nerve of the groupoid generated by a single isomorphism. The following proposition is an easy consequence of the definition of an equivalence:
	\begin{prop}
		\label{ext-inc-J}
	 If an edge $y$ in a simplicial set $S$ is an equivalence in $S$ then the morphism $y:\Delta[1] \to S$ can be extended along the inclusion $i$ i.e there is a (dashed) lifting arrow in the following (solid arrow) diagram:
	\[
	\xymatrix{
	\Delta[1] \ar[r]^y \ar[d]_i & S \\
	J \ar@{-->}[ru]_{u(y)}
}
	\]
	\end{prop}

\begin{df}
Let $(S, \E)$ be a marked simplicial set, then the marked arrows determine a simplicial map $\E \times \Delta[1] \to S$. The \emph{localization} of $(S, \E)$ is a simplicial set $S[\inv{\E}]$ defined by the following pushout diagram:
\begin{equation}
\label{local-sSets}
\xymatrix{
\E \times \Delta[1] \ar[r] \ar[d]_{\E \times i} & S \ar[d] \\
\E \times J \ar[r] & S[\inv{\E}]
}
\end{equation}

\end{df}
A charateristic property of the localization $S[\inv{\E}]$ is that for any simplicial map $G:S \to T$ which maps every marked edge in $\E$ to an equivalence in $T$, there exists a unique simplicial map $U:S[\inv{\E}] \to T$ such that the following diagram commutes:
\[
\xymatrix{
	\E \times \Delta[1] \ar[r] \ar[d]_{\E \times i} & S \ar[d]^p \ar@/^1pc/[rdd]^G \\
	\E \times J \ar[r] \ar@/_1pc/[rrd]_{u(\E)} & S[\inv{\E}] \ar@{-->}[rd]^U \\
	& & T
}
\]
The simplicial map $u(\E)$ in the above diagram is determined by the extension map from proposition \ref{ext-inc-J}.
\begin{nota}
	For a marked simplicial set $(S, \E)$, we denote by $(S[\inv{\E}], p(\E))$ the marked simplicial set whose set of marked arrows is the image of $\E$ under the projection map $p:S \to S[\inv{\E}]$. We observe that the set $p(\E)$ is a subset of the set of equivalences in $S[\inv{\E}]$ and therefore we have an inclusion map $(S[\inv{\E}], p(\E)) \subseteq (S[\inv{\E}], Eq)$.
	\end{nota}
 
 \begin{lem}
 	\label{inc-acy-cof}
 	The inclusion map $\iota_{(S, \E)}:(S[\inv{\E}], p(\E)) \subseteq (S[\inv{\E}], Eq)$ is an acyclic cofibration in $\sSetsQ$.
 	\end{lem}
 \begin{proof}
 	In light of \cite[Prop. 4.22]{sharma} it is sufficient to show that $\iota_{(S, \E)}$ has the left lifting property with respect to all fibrations between fibrant objects in $\sSetsQ$. Let $q:(X, Eq) \to (Y, Eq)$ be such a fibration. We want to show that whenever we have the following (solid) commutative diagram in $\sSetsM$, there exists a lifting arrow:
 		\[
 	\xymatrix{
 		(S[\inv{\E}], p(\E)) \ar[r] \ar[d]_{\iota_{(S, \E)}} & (X, Eq) \ar[d]^{q} \\
 		(S[\inv{\E}], Eq) \ar[r]  & (Y, Eq)
 	}
 	\]
 	We recall that the underlying simplicial map of $\iota_{(S, \E)}$ is the identity map.
 	Applying the forgetful functor $U$ to the above diagram, we get the following (outer) commutative diagram in $\sSets$ wherein the existence of the lifting arrow is obvious:
 	\[
 	\xymatrix{
 	(S[\inv{\E}]) \ar[r] \ar@{=}[d] &X \ar[d]^{U(q)} \\
 	(S[\inv{\E}]) \ar[r] \ar@{-->}[ru] & Y
    }
 	\]
 	By proposition \ref{Smap-Eq-to-Eq}, each simplicial morphism maps equivalences to equivalences which implies that the following (solid) commutative diagram has a lifting arrow:
 	\[
 \xymatrix{
 	(S[\inv{\E}], p(\E)) \ar[r] \ar[d]_{\iota_{(S, \E)}} & (X, Eq) \ar[d]^{q} \\
 	(S[\inv{\E}], Eq) \ar[r] \ar@{-->}[ru] & (Y, Eq)
 }
 \]
 Thus we have shown that the map $\iota_{(S, \E)}$ is an acyclic cofibration in $\sSetsMQ$.
 	\end{proof}

The above localization defines a functor $L:\sSetsM \to \sSets$. This functor has a right adjoint $E:\sSets \to \sSetsM$ which maps a simplicial $S$ to a marked simplicial set $(S, Eq)$ where $Eq$ is the set of equivalences in $S$. For each marked simplicial set $(S, \E)$, the unit map of this adjunction is the following composite:
\[
  (S, \E) \overset{p}\to (S[\inv{\E}], p(\E)) \subseteq (S[\inv{\E}], Eq) = EL((S, \E))
\]  
 where $p$ is the projection map. The counit map is an isomorphism.

\begin{prop}
	For each marked simplicial set $(S, \E)$ the projection map $p:S \to S[\inv{\E}]$ is a cofibration.
	\end{prop}
\begin{proof}
	The inclusion map $i:\Delta[1] \to J$ is a cofibration of simplicial sets. Now the proposition follows from the observation that cofibrations are preserved under cobase changes in a model category.
	\end{proof}
	We observe that the pushout square \eqref{local-sSets}, in $\sSets$, determines the following commutative square in the category of marked simplicial sets:
	\begin{equation}
	\label{mar-local}
	\xymatrix{
		\Sh{\E} \times \Sh{\Delta[1]} \ar[r] \ar[d]_{\E \times \Sh{i}} & (S, \E) \ar[d]^p \\
		\Sh{\E}  \times \Sh{J} \ar[r] & (S[\inv{\E}], p(\E))
	}
	\end{equation}
	The above proposition holds more strongly in the category of marked simplicial sets. We want to show that the unit map of the above adjunction is an acyclic cofibration in $\sSetsMQ$.

\begin{prop}
	For each marked simplicial set $(S, \E)$, the  commutative square \eqref{mar-local}, in $\sSetsM$, is a pushout square
	\end{prop}
\begin{proof}
	Let $(T, \Delta)$ be another marked simplicial set such that we have the following (outer) commutative digram:
	\begin{equation}
	\label{mar-pushout}
	\xymatrix{
		\Sh{\E} \times \Sh{\Delta[1]} \ar[r] \ar[d]_{\E \times i} & (S, \E) \ar[d]^p \ar@/^1pc/[rdd]^G \\
		\Sh{\E}  \times \Sh{J} \ar[r] \ar@/_1pc/[rrd]_{u(\E)} & (S[\inv{\E}], p(\E)) \ar@{-->}[rd]^L \\
		& & (T, \Delta)
	}
	\end{equation}
	We want to show the existence of the (dotted) lifting arrow $L$ which makes the entire diagram commutative.
	In this situation, the maps $G$ and $u(\E)$ take values in the (marked) simplicial subset $(T, Eq \cap \Delta) \subseteq (T, \Delta)$. Applying the forgetful functor $U$ to \eqref{mar-local} we get the  commutative diagram \eqref{local-sSets} in $\sSets$ which is a pushout diagram, therefore there exists a (dotted) arrow $L$ in the following (outer) commutative diagram in $\sSets$:
	\[
	\xymatrix{
		\E \times \Delta[1] \ar[r] \ar[d]_{\E \times i} & S \ar[d]^{U(p)} \ar@/^1pc/[rdd]^{U(G)} \\
		\E \times J \ar[r] \ar@/_1pc/[rrd]_{U(u(\E))} & S[\inv{\E}] \ar@{-->}[rd]^{L} \\
		& & T
	}
	\]
 The commutativity of this (simplicial) diagram implies that $L$ maps the set of edges $p(\E)$ into $(T, Eq \cap \Delta)$. This imples that $L$ is the desired dotted arrow in the commutative diagram \eqref{mar-pushout} in $\sSetsM$.
	
	\end{proof}
\begin{coro}
	\label{unit-Eq}
	The unit map of the adjunction $(L, E)$ is a (natural) acyclic cofibration.
	\end{coro}
\begin{proof}
	Since $\Sh{\E}$ is a discrete (marked) simplicial set therefore the commutative square \eqref{mar-local} is can be rewritten as follows:
	\begin{equation*}
	\xymatrix{
		\underset{\E} \sqcup \ \Sh{\Delta[1]} \ar[r] \ar[d]_{\underset{\E} \sqcup \ \Sh{i}} & (S, \E) \ar[d]^p \\
		\underset{\E} \sqcup \ \Sh{J} \ar[r] & (S[\inv{\E}], p(\E))
	}
	\end{equation*}
	The morphism $\Sh{i}$ is an acyclic cofibration in $\sSetsMQ$, it follows from\cite[Prop. 7.2.12]{Hirchhorn} that $\underset{\E} \sqcup \ \Sh{i}$ is also an acyclic cofibration in $\sSetsMQ$. The cobase change of an acyclic cofibration is again an acyclic cofibration which means that projection map $p$ in \eqref{mar-local} is an acyclic cofibration.
	Now the corollarly follows from lemma \ref{inc-acy-cof}.
	\end{proof}

\begin{prop}
	The adjunction $(L, E)$ induces an (adjoint) equivalence between the homotopy category of  $\sSetsMQ$ and that of $\sSetsQ$.
	\end{prop}
\begin{proof}
	We have seen above that both the unit and the counit maps of the adjunction $(L, E)$ are (natural) weak equivalences therefore, in light of \cite[Prop. ]{DHKS} it is sufficient to show that $L$ and $E$ are homotopical (weak-equivalence preserving) functors. Let $F:S \to T$ be a weak equivalence in $\sSetsQ$. We want to show that $E(F):E(S) \to E(T)$ is a weak equivalence in $\sSetsMQ$. Since the counit map of the adjunction in context is an isomorphism therefore the right adjoint $E$ is fully faihful. We observe that for all $n \in \Nat$, $E(S \times \Delta[n]) = E(S) \times \Fl{\Delta[n]}$. These two facts together imply that for any simplicial set $W$, we have a natural isomorphism:
	\[
	[S, W] \cong \Fl{[(S, Eq),(W, Eq)]} =\Fl{[E(S), E(W)]}.
	\]
	We observe that for any fibrant marked simplicial set $Z$ we have the following equality: $Z = EU(Z)$, where $U$ is the forgetful functor. Now the following commutative diagram implies that $E$ preserves weak-equivalences:
	\begin{equation*}
	\xymatrix{
	 [S, U(Z)] \ar[r]  & \Fl{[E(S), EU(Z)]}  \ar@{=}[r] & \Fl{[E(S), Z]}  \\
	 [T, U(Z)] \ar[r] \ar[u]^{[F, U(Z)]}  & \Fl{[E(T), EU(Z)]} \ar@{=}[r] \ar[u] & \Fl{[E(T), Z]} \ar[u]_{[E{F}, Z]}
    }
	\end{equation*}
	
	Let $G:(S, \E) \to (T, \Delta)$ be a weak equivalence in $\sSetsMQ$. The unit natural transformation $\eta$ provides us the following commutative diagram:
	\[
	\xymatrix{
	(S, \E) \ar[r]^G \ar[d]_{\eta_{(S, \E)}} & (T, \Delta) \ar[d]^{\eta_{(T, \Delta)}} \\
	EL((S, \E)) \ar[r]_{EL(G)} & EL((T, \Delta))
    }
	\]
	The vertical arrows are acyclic cofibrations by corollary \ref{unit-Eq} and the top horizontal arrow is a weak equivalence by assumption, therefore the two out of three property of weak-equivalences in a model category implies that $EL(G)$ is a weak-equivalence in $\sSetsMQ$. Let $Y$ be any quasi-category, then we have the following commutative square:
	\begin{equation*}
	\xymatrix{
		[L((S, \E)), Z] \ar[r]^\cong   & \Fl{[EL((S, \E)), E(Z)]}  \\
		[L((T, \Delta)), Z] \ar[r]_\cong \ar[u]^{[L(G), Z]}  & \Fl{[EL((T, \Delta)), E(Z)]} \ar[u]_{[EL(G), E(Z)]}
	}
	\end{equation*}
	The two horizontal arrows are isomorphisms and the right downward arrow is a weak-equivalence. The two out of three property of weak equivalences in a model category implies that for each quasi-category $Z$, the simplicial map $[L(G), Z]$ is a weak equivalence and therefore $L(G)$ is a weak equivalence in $\sSetsQ$. Thus we have shown that $L$ preserves weak equivalences.
	
	\end{proof}
	 For any (small) category $D$, the adjunction $(L, E)$ induces an adjunction 
	 \begin{equation*}
	 L^D:[D, \sSets^+] \rightleftharpoons [D, \sSets]:E^D.
	 \end{equation*}
	 Each of the two functor categories can be endowed with a \emph{projective} model category structure which is inherited from $\sSetsMQ$ and $\sSetsQ$ respectively, see \cite[Remark 2.8.6]{JL}.
	 \begin{nota}
	 	We will denote the aforementioned model categories by $[D, \sSetsMQ]$ and $[D, \sSetsQ]$ respectively.
	 	\end{nota}
 	\begin{nota}
 		We will denote the homotopy categories of $[D, \sSetsMQ]$ and $[D, \sSetsQ]$ by $Ho(\sSetsMQ^D)$ and $Ho(\sSetsQ^D)$ respectively.
 		\end{nota}
	\begin{coro}
	The adjunction $(L^D, E^D)$ induces an (adjoint) equivalence of categories on the homotopy categories:
	\begin{equation*}
	Ho(L^D):Ho(\sSetsMQ^D) \rightleftharpoons Ho(\sSetsQ^D):Ho(E^D).
	\end{equation*}
	\end{coro}
\begin{proof}
	Both adjoint functors $L^D$ and $E^D$ are homotopical functors and therefore both are deformable in the sense of \cite[Def. ]{DHKS} with identity deformations. Now the result follows from \cite[Prop. 45.2]{DHKS}.

	\end{proof}
 We denote the colimit functor on the functor category $[D, \sSetsM]$ by
 \[
 {\varinjlim}^+:[D, \sSetsM] \to \sSetsM
 \]
 Consider the following composite functor:
 \begin{equation}
 \label{Eq-colim}
 [D, \sSets] \overset{E^D} \to [D, \sSetsM] \overset{\colim{+}} \to \sSetsM \overset{L} \to \sSets
 \end{equation}
 which we denote by $L\colim{}$. The next lemma shows that this functor is a colimit functor on $\sSets$:
 \begin{lem}
 	\label{colim-in-QCat}
 	The composite functor $L\colim{}$, \eqref{Eq-colim}
 	is a colimit functor.
 	\end{lem}
 \begin{proof}
 	Let $F:D \to \sSetsQ$ be a functor and $Y$ be a simplicial set.
 	It is sufficient to observe the following chain of (natural) bijections:
 	\begin{multline*}
 	\sSets(L(\colim{}(E^D(F))), Y) \cong \sSetsM(\colim{+}(E^D(F)), E(Y)) \cong [D, \sSetsM](E^D(F), \Delta(E(Y))) = \\
 	[D, \sSetsM](E^D(F), E^D(\Delta(Y))) \cong [D, \sSets](L^D(E^D(F)), \Delta(Y)) \cong \\ \sSets(\colim{}(L^D(E^D(F)), Y) \cong \sSets(\colim{}(F), Y).
 	\end{multline*}
 	The last bijection in the above chain follows from the observation that the counit of the adjunction $(L^D, E^D)$ is an isomorphism.
 
 	\end{proof}
 We now consider the following composite functor:
 \begin{equation}
 \label{Eq-hocolim}
 [D, \sSets] \overset{E^D} \to [D, \sSetsM] \overset{u_! \circ \gCLM{D}} \to \sSetsM \overset{L} \to \sSets
 \end{equation}
 which we denote by $L\hColim{}$.
 Now we are ready to state and prove the main result of this secion:
 \begin{thm}
 	The functor $L\hColim{}$ is a homotopy colimit functor.
 	\end{thm}
 \begin{proof}
 	The functor $L\hColim{}$ is a composite of three homotopical functors therefore it is homotopical. Now we consider the following chain of (natural) bijections for a pair of objects $F \in [D, \sSets]$ and $Y \in \sSets$:
 	\begin{multline*}
 	Ho_{\sSets}(L\hColim{}(F), Y) \cong Ho_{\sSetsM}(u_!\gCLM{D}E^D(F), E(Y)) \cong \\
 	Ho_{\sSetsMQ^D}(E^D(F), \Delta(E(Y))) = Ho_{\sSetsMQ^D}(E^D(F), E^D(\Delta(Y))) \cong \\
 	Ho_{\sSetsQ^D}(L^DE^D(F), \Delta(Y)) \cong Ho_{\sSetsQ^D}(F, \Delta(Y))
 	\end{multline*}

 	\end{proof}

 \appendix
 \section[Marked simplicial sets]{A review of marked simplicial sets}
\label{mar-sSets}
In this appendix we will review the theory of marked simplicial sets. Later in this paper we will develop a theory of coherently commutative monoidal objects in the category of marked simplicial sets.

\begin{df}
	\label{mar-sSet}
	A \emph{marked} simplicial set is a pair $(X, \E)$, where $X$ is a simplicial set and $\E$ is a set of edges of $X$ which contains every degenerate edge of $X$. We will say that an edge of $X$ is \emph{marked} if it belongs to $\E$.
	A morphism $f:(X, \E) \to (X', \E')$ of marked simplicial sets is a simplicial map $f:X \to X'$ having the property that $f(\E) \subseteq \E'$. We denote the category of marked simplicial sets by $\sSetsM$.
	\end{df}

Every simplicial set $S$ may be regarded as a marked simplicial set in many ways. We mention two extreme cases: We let $\Sh{S} = (S, S_1)$ denote the marked simplicial set in which every edge is marked. We denote by $\Fl{S} = (S, s_0(S_0))$ denote the marked simplicial set in which only the degerate edges of $S$ have been marked.

The category $\sSetsM$ is \emph{cartesian-closed}, \emph{i.e.} for each pair of objects $X, Y \in Ob(\sSetsM)$, there is an internal mapping object $[X, Y]^+$ equipped with an \emph{evaluation map} $[X, Y]^+ \times X \to Y$ which induces a bijection:
\[
\sSetsM(Z, [X, Y]^+) \overset{\cong} \to \sSetsM(Z \times X, Y),
\]
for every $Z \in \sSetsM$.
\begin{nota}
	We denote by $\Fl{[X, Y]}$ the underlying simplicial set of $[X, Y]^+$.
	\end{nota}
The mapping space $\Fl{[X, Y]}$ is charaterized by the following bijection:
\[
\sSets(K, \Fl{[X, Y]}) \overset{\cong} \to \sSetsM(\Fl{K} \times X, Y),
\]
for each simplicial set $K$.
\begin{nota}
	We denote by $\Sh{[X, Y]}$ the simplicial subset of $\Fl{[X, Y]}$ consisting of all simplices $\sigma \in \Fl{[X, Y]}$ such that every edge of $\sigma$ is a marked edge of $[X, Y]^+$.
\end{nota}
The mapping space $\Sh{[X, Y]}$ is charaterized by the following bijection:
\[
\sSets(K, \Sh{[X, Y]}) \overset{\cong} \to \sSetsM(\Sh{K} \times X, Y),
\]
for each simplicial set $K$.

 The Joyal model category structure on $\sSets$ has the following analog for marked simplicial sets:
 \begin{thm}
 	\label{Joyal-sSetsM}
 	There is a left-proper, combinatorial model category structure on the category of marked simplicial sets $\sSetsM$ in which a morphism $p:X \to Y$ is a
 	\begin{enumerate}
 		\item cofibration if the simplicial map between the underlying simplicial sets is a cofibration in $\sSetsQ$, namely a monomorphism.
 		
 		\item a weak-equivalence if the induced simplicial map on the mapping spaces
 		\[
 		\Fl{[p, \Nt{K}]}:\Fl{[X, \Nt{K}]} \to \Fl{[Y, \Nt{K}]}
 		\]
 		is a weak-categorical equivalence, for each quasi-category $K$.
 		
 		\item fibration if it has the right lifting property with respect to all maps in $\sSetsM$ which are simultaneously cofibrations and weak equivalences.
 		
 		\end{enumerate}
 	Further, the above model category structure is enriched over the Joyal model category, i.e. it is a $\sSetsQ$-model category.
 	\end{thm}
 The above theorem follows from \cite[Prop. 3.1.3.7]{JL}.
 \begin{nota}
 	We will denote the model category structure in Theorem \ref{Joyal-sSetsM} by $\sSetsMQ$ and refer to it either as the \emph{Joyal} model category of \emph{marked} simplicial sets or as the model category of marked quasi-categories.
 	\end{nota}
 \begin{thm}
 	\label{Cart-cl-Mdl-S-plus}
 	The model category $\sSetsMQ$ is a cartesian closed model category.
 	\end{thm}
 \begin{proof}
 	The theorem follows from \cite[Corollary 3.1.4.3]{JL} by taking $S = T = \Delta[0]$.
 	\end{proof}
 
 There is an obvious forgetful functor $U:\sSetsM \to \sSets$. This forgetful functor has a left adjoint $\Fl{(-)}:\sSets \to \sSetsM$.
 \begin{thm}
 	\label{Quil-eq-JQ-MJQ}
 	The adjoint pair of functors $(\Fl{(-)}, U)$ determine a Quillen equivalence between the Joyal model category of marked simplicial sets and the Joyal model category of simplicial sets.
 	\end{thm}
 The proof of the above theorem follows from \cite[Prop. 3.1.5.3]{JL}.
 \begin{rem}
 	A marked simplicial set $X$ is fibrant in $\sSetsMQ$ if and only if it is a quasi-category with the set of all its equivalences as the set of marked edges.
 	\end{rem}
 The following proposition brings out a very important distinction between the model category $\sSetsMQ$ and $\sSetsQ$:
 \begin{prop}
 	\label{inc-J-acy-cof}
 	The inclusion map
 	\[
 	\Sh{i}:\Sh{\Delta[1]} \to \Sh{J}
 	\]
 	is an acyclic cofibration in the model category $\sSetsMQ$.
 	\end{prop}
 \begin{proof}
 	In light of \cite[Prop. 4.22]{sharma} it is sufficient to show that $\Sh{i}$ has the left lifting property with respect to every fibration between fibrant objects in $\sSetsMQ$. Let $p:X \to Y$ be such a fibration. We observe that $p = EU(p)$. By adjointness the aforementioned left lifting property is equivalent to $U(p)$ having left lifting property with respect to $L(\Sh{i})$. The proposition now follows from the observation that the map $L(\Sh{i})$ is a canonical isomorphism in $\sSetsQ$.
 	\end{proof}
 The above proposition does NOT hold in $\sSetsQ$ \emph{i.e.} the inclusion map $i:\Delta[1] \to J$ is a cofibration but it is NOT acyclic in $\sSetsQ$.

 \bibliographystyle{amsalpha}
\bibliography{GrConsQCat}

\end{document}